\documentclass{amsart}

\usepackage{amssymb,latexsym,amsfonts,amsmath}
\usepackage{graphicx,epsfig}
\include{diagrams}
\usepackage{graphicx}
\usepackage{color}
\usepackage{diagrams}
\usepackage{amsfonts}
\usepackage{amssymb,latexsym,amsfonts,amsmath}
\usepackage{graphicx}
\usepackage{tikz}
\usepackage{subfigure}
\usetikzlibrary{automata}

\usepackage{savesym}
\savesymbol{AND}
\usepackage{algorithmic}
\usepackage{algorithm}

\newtheorem{theorem}{Theorem}[section]
\newtheorem{lemma}[theorem]{Lemma}

\newtheorem{proposition}[theorem]{Proposition}
\newtheorem{corollary}[theorem]{Corollary}

\theoremstyle{definition}
\newtheorem{definition}[theorem]{Definition}
\newtheorem{example}[theorem]{Example}

\theoremstyle{remark}
\newtheorem{remark}[theorem]{Remark}

\numberwithin{equation}{section}

\newcommand{\kr}{\textrm{ker}}

\renewcommand{\dim}{\mathrm{dim}}

\newcommand{\im}{\mathrm{Im}}
\newcommand{\rank}{\mathrm{rank}}

\newcommand{\supp}{\mathrm{supp}}
\newcommand{\ext}{\mathrm{e}}
\newcommand{\bis}{\mathrm{b}}
\newcommand{\lin}{\textit{l}}

\newcommand{\Reach}{\mathrm{Reach}}
\newcommand{\Obs}{\mathrm{Obs}}
\newcommand{\diag}{\mathrm{diag}}
\newcommand{\col}{\mathrm{col}}
\newcommand{\row}{\mathrm{row}}

\newcommand{\Gcal}{\mathcal{G}}
\newcommand{\Ical}{\mathcal{I}}

\newcommand{\Qcal}{\mathcal{Q}}
\newcommand{\Rcal}{\mathcal{R}}
\newcommand{\Scal}{\mathcal{S}}
\newcommand{\Ucal}{\mathcal{U}}

\def\Real{\mathbb{R}}
\def\Compl{\mathbb{C}}

\usepackage{bm}
\def\xb{\bm{x}}  \def\yb{\bm{y}} \def\ub{\bm{u}} \def\wb{\bm{w}} 
\def\wbtilde{\bm{\tilde{w}}}  
\def\mub{\boldsymbol{\mu}} \def\nub{\boldsymbol{\nu}} 
\def\zerob{\mathbf{0}}

\def\Sigmatilde{{\widetilde{\Sigma}}} 
\def\Atilde{{\widetilde{A}}}  \def\Btilde{{\widetilde{B}}}
\def\Ctilde{{\widetilde{C}}}  
\def\ntilde{{\tilde{n}}}

\def\cov{\textrm{cov}}

\def\Ubar{\,\overline{\!U}}

\begin{document}

\begin{abstract}                          
In this paper we propose definitions of equivalence via stochastic bisimulation and of equivalence of stochastic external behavior
for the class of discrete--time stochastic linear control systems with possibly degenerate normally distributed disturbances. The first notion is inspired by the notion of probabilistic bisimulation for probabilistic chains while the second one by the notion of equivalence of external behavior for (nonstochastic) behavioral systems. 
Geometric necessary and sufficient conditions for checking these notions are derived. Model reduction via Kalman--like decomposition is also proposed. Connections with stochastic linear realization theory and stochastic reachability are established. 
\end{abstract}

\title[Equivalence Notions for Discrete--Time Stochastic Linear Control Systems]{Equivalence Notions for Discrete--Time \\Stochastic Linear Control Systems}

\thanks{The research leading to these results has been partially supported by the Center of Excellence DEWS.}

\author[Giordano Pola, Costanzo Manes, Arjan J. van der Schaft and Maria Domenica Di Benedetto
]{Giordano Pola, Costanzo Manes, Arjan J. van der Schaft and Maria Domenica Di Benedetto}
\address{$^1$Department of Information Engineering, Computer Science and Mathematics, Center of Excellence DEWS,
University of L{'}Aquila, 67100 L{'}Aquila, Italy}
\email{
\{giordano.pola,costanzo.manes,mariadomenica.dibenedetto\}@univaq.it}
\address{$^2$Johann Bernoulli Institute for Mathematics and Computer Science, University of Groningen, P.O. Box 407, 9700 AK Groningen, The Netherlands}
\email{a.j.van.der.schaft@rug.nl.}

\maketitle

\section{Introduction} \label{sec1}

A theme widely studied in the community of computer science is the characterization of equivalent models of computation. Several equivalence notions have been proposed in the literature, see e.g. \cite{Spectrum} and the references therein. Among these notions, bisimulation \cite{Milner,Park} and trace equivalences play a prominent role. 
As discussed in \cite{Spectrum}, trace equivalence is the weakest equivalence notion, while bisimulation equivalence is the strongest one, apart from the notion of equivalence via isomorphism. 
Bisimulation equivalence is extensively used in the community of computer science as an effective tool to mitigate software verification. In the last thirty years, many researchers in the control systems and computer science communities were attracted by this research topic with the aim of reducing the complexity of real--world complex systems for formal verification/analysis and control design purposes. The research in this field is very broad and can be roughly categorized along the following directions:
\begin{itemize}
\item (D) Type of dynamics considered: deterministic/non--deterministic (D1), versus, stochastic (D2);
\item (R) Type of reduction obtained: reduction of a finite states model to a smaller finite states model (R1), versus, reduction of a continuous/hybrid (infinite states) model to a finite states model (R2), versus, reduction of a continuous/hybrid (infinite states) model to a smaller (with lower dimensional state space) continuous/hybrid (infinite states) model (R3);
\item (E) Type of equivalence notions employed: exact simulation/bisimulation/trace equivalence notions (E1), versus, approximate simulation/bisimulation/trace equivalence notions (E2).
\end{itemize}

A (non--exhaustive) list of literature relevant in this research topic is reported in Table \ref{Tab}. The present paper is along the research line (D2)--(R3)--(E1) and aims at extending the theory of bisimulation and external behavior equivalences given for non--deterministic control systems in \cite{BisimSchaft} to a stochastic setting. As discussed later on in the paper, the proposed notions have been inspired by the corresponding notions given in the finite systems domain (D2)--(R1)--(E1) and for behavioral systems \cite{Willems}. We briefly recall that within (D2)--(R1)--(E1), bisimulation equivalence for probabilistic chains has been introduced in \cite{Larsen:91}; a generalization of this notion to Labelled Markov Processes has been studied in \cite{Desharnais2004}, and to Interactive Markov Chains, mixing transitions due to interaction with spontaneous probabilistic transitions, in \cite{HermannsBook}. 
Within the research line (D2)--(R3)--(E1) where the present paper is placed, a notion of bisimulation for general stochastic hybrid systems (with no inputs and outputs) has been proposed in \cite{BujorianuHSCC05} and for communicating piecewise deterministic markov processes in \cite{StrubbeHSCC05,StrubbeHybridge}. However, given the generality of the models considered in \cite{BujorianuHSCC05,StrubbeHSCC05,StrubbeHybridge}, checkable conditions for verifying bisimulation equivalence are difficult to find. 
For this reason in this paper we consider a simpler class of stochastic control systems and propose equivalence notions that can be effectively checked. \\
We consider the class of discrete--time stochastic linear control systems with possibly degenerate normally distributed disturbances and propose the notions of equivalence via stochastic bisimulation and equivalence of stochastic external behavior. 
Comparisons of the first notion with the ones proposed in \cite{Larsen:91,Desharnais2004, BujorianuHSCC05,StrubbeHSCC05,StrubbeHybridge} are  discussed in the paper. 
The first notion is formally proven to imply the latter, while the converse implication is shown to be not true by means of a counterexample. 
Necessary and sufficient conditions to check this notion in terms of geometric control theory are derived and model reduction discussed. 
The concept of stochastic reachability, see e.g. \cite{BujorianuHSCC2003, AbateAut, PolaMTNS2006}, is related to the notion of stochastic bisimulation. The proposed notion of equivalence via stochastic bisimulation preserves stochastic reachability properties of the systems involved. This is important because, as outlined in the paper, control strategies designed to solve some stochastic reachability--based specifications can be readily transferred between systems that are equivalent via stochastic bisimulation. This result extends well known facts for (finite states) concurrent processes, see e.g. \cite{ModelChecking}, to stochastic systems with infinite number of states. 
Connections with stochastic realization theory, 
see \cite{Kalman1,Faurre} and also \cite{shaftrealization}, are also established. 
A preliminary version of this paper appeared in the conference publication \cite{PolaCDC2015}. 
The present paper extends the work \cite{PolaCDC2015} by introducing novel results on model reduction and on connections with stochastic realization theory. Finally, it also includes proofs of all the results. The problem addressed in this paper has been recently investigated in the continuous--time domain in \cite{PolaCDC16b,PolaCDC16a}.
\\
The paper is organized as follows. In Section \ref{sec2} we introduce the notation employed and recall preliminary definitions. In Section \ref{sec3} we present the notions of equivalence via stochastic bisimulation and of equivalence of stochastic external behavior; connections with stochastic reachability are also discussed. In Section \ref{sec4} we provide geometric conditions for checking the proposed notions. Model reduction is discussed in Section \ref{sec5}. 
In Section \ref{sec7} we discuss connections with notions of bisimulation for probabilistic chains and labelled Markov processes 
and with stochastic linear realization theory. Finally, Section \ref{sec8} offers some concluding remarks and outlook.

\begin{table}
\begin{center}
\begin{tabular}
[c]{l|l|l}\hline
						& (D1) 																														& (D2) 															\\\hline
(R1)--(E1)  & \cite{ModelChecking} 																						& \cite{Larsen:91,HermannsBook}	\\\hline
(R1)--(E2)	& \cite{Caspi2002,DiPierro2003,DeAlvaro2004}											& \cite{Desharnais2004,Breugel03anintrinsic} \\\hline
(R2)--(E1)	& \cite{paulo,DiscAbs,Lee2012,BeltaTACTN2010,Belta2012,junge1,gunther}	& \cite{Lunze2001,LunzeBook}				\\\hline
(R2)--(E2)	& \cite{paulo,GirardEJC11,DiBenede:HAS2013,PolaPWA} 							& \cite{ZamaniSCL2014,ZamaniTAC2014,AbateTAC2011}							\\\hline
(R3)--(E1)	& \cite{BisimSchaft,HTP02,PolaIJC06} 															& \cite{BujorianuHSCC05,StrubbeHSCC05,StrubbeHybridge,Blute:97} 							\\\hline
(R3)--(E2)	& \cite{AB-TAC07}																									& \cite{JuliusTAC2009}		\\\hline
\end{tabular}
\caption{Related literature on equivalences based reduction of discrete/continuous/hybrid deterministic/non--deterministic/stochastic systems.}
\label{Tab}
\end{center}
\end{table}

\section{Notation and preliminary definitions} \label{sec2}

Given a pair of sets $S_1$ and $S_2$ and a relation $\Rcal \subseteq S_1 \times S_2$, we define 
for any sets $X_1 \subseteq S_1$ and $X_2 \subseteq S_2$, 
$\Rcal(X_1)=\{x_2\in S_2 | \exists x_1 \in X_1 \text{ s.t. } (x_1,x_2)\in \Rcal \}$ and 
$\Rcal^{-1}(X_2)=\{x_1\in S_1 | \exists x_2 \in X_2 \text{ s.t. } (x_1,x_2)\in \Rcal \}$. 
Relation $\Rcal$ is total if $\Rcal(S_1)=S_2$ and $\Rcal^{-1}(S_2)=S_1$. 
The standard symbols $\mathbb{N}$, $\Real$ and $\Real^+$ denote the sets of nonnegative integer, real, and positive real numbers, respectively. 
Given a vector $x\in\Real^n$, the symbol $x[i]$  
denotes the $i$--th component of $x$.   
Given a matrix $M\in\Real^{m\times n}$, the symbols $M^T$, $\rank(M)$, $\im(M)$ and $\ker(M)$ denote the transpose, the rank, the image and the kernel of $M$, respectively.
If $M$ is square, $\det(M)$ denotes the determinant of $M$. 
Given a subset $X$ of $\Real^n$ we denote by $MX$ the image of $X$ through $M$, i.e. the set $\{y\in\Real^m|\exists x\in X \text{ s.t. } y=Mx\}$. 
The symbols $I_{n}$ and $0_{n\times m}$ denote the $(n,n)$--identity matrix and the $(n,m)$--null matrix, respectively; the symbol $0_n$ denotes the null vector in $\Real^n$. 
Given a collection of square matrices $M_{1},M_{2},...,M_{N}$, we denote by $\diag(M_{1},M_{2},...,M_{N})$ the block--diagonal matrix with block--entries $M_i$. 
The symbol $\oplus$ denotes the direct sum operator between subspaces. 
Given $f_i:\mathbb{N}\rightarrow \Real^n$, $i=1,2$, we write $f_1 = f_2$ instead of $f_1(t)=f_2(t)$ for all $t\in\mathbb{N}$ and also 
$f = \mathbf{0}$ instead of $f(t)=0$ for all $t\in\mathbb{N}$.
\\
Consider a probability space $(\Omega,\mathcal{F},\mathbf{P})$. 
$\mathbf{P}(S_1 | S_2)$ denotes the conditional probability of event $S_1$ given event $S_2$. 
Given a random variable $\xb:\Omega\rightarrow \Real^n$ and a measurable set $X\subseteq \Real^n$, we use standard shorthand notation $\mathbf{P}(\xb\in X)$ instead of $\mathbf{P}(\{\omega\in\Omega:\xb(\omega)\in X\})$; we denote by $\supp(\xb)$ the support of $\xb$; we recall that:
\[
\supp(\xb)=\{z\in\Real^n | \mathbf{P}(\xb \in \mathcal{B}_{\rho}(z))>0, \text{ for any } \rho\in\Real^+\},
\]
where $\mathcal{B}_{\rho}(z)=\{z\in\Real^n|\Vert z \Vert<\rho\}$. 
In this paper we consider random variables whose support is a manifold. 
Random variable $\xb$ is degenerate if $\dim(\supp(\xb))<n$ and non--degenerate, otherwise. 
Of course 
\begin{equation}
\label{supporto}
\mathbf{P}(\xb\in X)=\mathbf{P}(\xb\in X \cap \supp(\xb)).
\end{equation}
In general, the random variables considered in this paper are degenerate, and therefore
do not admit a probability density function.
$E(\xb)$ denotes the expected value of the r.v.\ $\xb$, and 
$\cov(\xb,\yb)$ denotes the covariance between two random vectors $\xb$ and $\yb$, i.e.\
$\cov(\xb,\yb)=E\big[\big(\xb-E(\xb)\big)\big(\yb-E(\yb)\big)^T\big]$.

The following standard definition will be used in this paper:  
\begin{definition}
Two stochastic processes $\xb_1:\mathbb{N}\times \Omega\rightarrow \Real^n$ and $\xb_2:\mathbb{N}\times \Omega\rightarrow \Real^n$ are stochastically  equivalent, denoted 
$
\xb_1 \sim \xb_2
$, 
if the probability distribution of the two vectors $(\xb_1(t_1),\xb_1(t_2),...,\xb_1(t_k))$ and $(\xb_2(t_1),\xb_2(t_2),...,\xb_2(t_k))$ is equal for all choices of times $t_1,t_2,...,t_k \in\mathbb{N}$. 
\end{definition}

The standard notation $\xb \sim \mathcal{N}(\mu,\Psi)$ indicates that $\xb$ is a random variable with normal distribution with mean vector $\mu$ and covariance matrix $\Psi$; we recall that $\Psi$ is symmetric and positive semi-definite, $\supp(\xb)=\mu+\im(\Psi)$ and $\xb$ is degenerate if $\det(\Psi)=0$ and non--degenerate, otherwise. 
Moreover, we recall that if $\xb \sim \mathcal{N}(\mu,\Psi)$ then $\mathbf{y}=\alpha \xb + \beta \sim \mathcal{N}(\beta+\alpha\mu,\alpha\Psi \alpha^T)$.

\section{Equivalence notions} \label{sec3}

In this section we propose the notions of equivalence of stochastic external behavior and equivalence via stochastic bisimulation for a pair of discrete--time stochastic linear control systems $\Sigma_1$ and $\Sigma_2$
described, for $t\in\mathbb{N}$, by:
\begin{equation}
\label{eq:systemSi}
\Sigma_i:
\left\{
\begin{aligned}
x_i(t+1) & = A_i x_i(t) + B_i u_i(t) + G_i w_i(t) ,\\
y_i (t)  & = C_i x_i(t) + \nu_i(t),\\
x_i \in & \Real^{n_i}, u_i \in \Real^{m}, w_i \in \Real^{l_i}, y_i,\nu_i \in \Real^{p},
\end{aligned}
\right.
\end{equation}
where $x_i$ is the state, $u_i$ is the control input, $y_i$ is the output, and
$w_i$ and $\nu_i$ are random disturbances.
We assume that $\nu_i(t) \sim \mathcal{N}(0, \Psi_i )$, with $\Psi_i\in \Real^{p\times p}$,
and $w_i(t) \sim \mathcal{N}(\mu_i, W_i )$ with $\mu_i \in \Real^{l_i}$ and $W_i\in \Real^{l_i\times l_i}$. 
  We also assume that both sequences $w_i(t)$ and $\nu_i(t)$ are white and mutually independent.
  Without loss of generality we assume in the sequel that $W_i=I_{l_i}$, so that the resulting random vector
$v_i=G_i w_i$ is $v_i(t)\sim \mathcal{N}(G_i\mu_i, G_iG_i^T )$.
Note that $v_i(t)$ is degenerate if and only if $\rank(G_i)< n_i$.
In the following, the boldface symbols $\ub_i$, $\wb_i$ and $\nub_i$ 
will be used to denote the whole sequences of deterministic inputs $u_i(t)$ and random noises $w_i(t)$ and $\nu_i(t)$, $t\ge 0$, i.e. 
$\ub_i:\mathbb{N}\rightarrow \Real^m$, 
$\wb_i:\mathbb{N}\times \Omega \rightarrow \Real^{l_i}$, and
$\nub_i:\mathbb{N}\times \Omega \rightarrow \Real^{p}$. 
The state and output values of the system $\Sigma_i$ 
at times $t\in\mathbb{N}$ are computed as
\begin{align}
& \xb_i(t,x^0_i,\ub_i,\wb_i) 
 = A_i^t x_i^0 + \sum_{\tau=0}^{t-1} A_i^{t-1-\tau} ( B_i u_i(\tau)+ G_i w_i (\tau) ), \label{TFs}\\
& \yb_i(t,x^0_i,\ub_i,\wb_i,\nub_i) =C_i \xb_i(t,x^0_i,\ub_i,\wb_i) + \nu_i(t)  \notag\\
& = C_i A_i^t x_i^0  +\! \sum_{\tau=0}^{t-1} C_i A_i^{t-1-\tau}\!( B_i u_i(\tau)+ G_i w_i (\tau) ) + \nu_i(t). 
\end{align}
For a given initial condition $x^0_i \in\Real^{n_i}$ and deterministic
input $\ub_i$ let us denote by 
$\xb_i\vert_{x_i^0,\ub_i}: \mathbb{N}\times \Omega \rightarrow \Real^{n_i}$
and 
$\yb_i\vert_{x_i^0,\ub_i}: \mathbb{N}\times \Omega \rightarrow \Real^{p}$
the state and output stochastic processes generated by the system $\Sigma_i$
driven by the stochastic sequences $\wb_i$ and $\nub_i$, i.e.
\begin{align}
\xb_i\vert_{x_i^0,\ub_i}(t)& =\xb_i(t,x^0_i,\ub_i,\wb_i), \\
\yb_i\vert_{x_i^0,\ub_i}(t)& =\yb_i(t,x^0_i,\ub_i,\wb_i,\nub_i).
\end{align}
For later purposes let us define the sequences $\mub_i=E(\wb_i)$ of expected values 
of the random sequences $\wb_i$, $i=1,2$, 
and the sequence $\wbtilde_i=\wb_i-\mub_i$, that is the centered (i.e.\ zero mean) version of the disturbance $\wb_i$.
Then, by linearity
\begin{align}
\xb_i(t,x^0_i,\ub_i,\wb_i) & = \xb_i(t,x^0_i,\ub_i,\mub_i) + \xb_i(t,0,\mathbf{0},\wbtilde_i) \label{eq:xlin1} \\
\yb_i(t,x^0_i,\ub_i,\wb_i,\nub_i) & = \yb_i(t,x^0_i,\ub_i,\mub_i,\zerob) + \yb_i(t,0,\mathbf{0},\wbtilde_i,\nub_i) \label{eq:ylin1}
\end{align}

Defining for any time $t\in\mathbb{N}$ and $i=1,2$ the zero mean vectors
\begin{equation}
\label{Wvector}
\wbtilde_i\vert_{0:t-1}= 
\begin{bmatrix} 
\wb_i(t-1) - \mu_i\\
\vdots\\
\wb_i(1) - \mu_i\\
\wb_i(0) - \mu_i
\end{bmatrix}
\end{equation}
and the matrices
\begin{align} 
\Reach_t (A_i,G_i) & =\begin{bmatrix}
G_i & A_i G_i & ... & A_i^{t-1} G_i \end{bmatrix},\label{eq:reachdef}\\ 
\Obs_t (A_i,C_i) & = \big(\Reach_t (A_i^T,C_i^T) \big)^T, \label{eq:obsdef}
\end{align}
we can rewrite \eqref{eq:xlin1} and \eqref{eq:ylin1} as 
\begin{align}
& \begin{aligned}
\xb_i\vert_{x_i^0,\ub_i}(t) = & \,\xb_i(t,x^0_i,\ub_i,\mub_i) \\
						& + \Reach_t (A_i,G_i)\wbtilde_i\vert_{0:t-1} ,
\end{aligned}\label{eq:xlin2} \\
& \begin{aligned}
\yb_i\vert_{x_i^0,\ub_i}(t) = & \, \yb_i(t,x^0_i,\ub_i,\mub_i,\nub_i) \\
				& + C_i \Reach_t (A_i,G_i)\wbtilde_i\vert_{0:t-1} + \nu_i(t).
\end{aligned}\label{eq:ylin2}
\end{align}
Note that $C_i \Reach_t (A_i,G_i)=\Obs_t(A_i,C_i)\,G_i$.\\
If the initial state $x_i^0$ in $\Sigma_i$ is considered as deterministic, 
then the first terms of the right hand sides of equations  \eqref{eq:xlin2} and \eqref{eq:ylin2}
(or \eqref{eq:xlin1} and \eqref{eq:ylin1})
are the expected values of the state and output processes at time $t$.
    However, for the sake of generality, we look at these terms
as expectations conditional to $x_i^0$, considered as a random variable 
independent of both $\wb_i$ and $\nub_i$.
    Explicit expressions of the these terms are:
\begin{align}
& \xb_i(t,x^0_i,\ub_i,\mub_i) = E\{ \xb_i\vert_{x_i^0,\ub_i}(t)|x_i^0\}  \notag \\
 & = A_i^t x_i^0 + \sum_{\tau=0}^{t-1} A_i^{t-1-\tau} ( B_i u_i(\tau)+ G_i \mu_i ), \label{eq:ximean}\\
& \yb_i(t,x^0_i,\ub_i,\mub_i,\zerob) = E\{ \yb_i\vert_{x_i^0,\ub_i}(t)|x_i^0\} = C_i \xb_i(t,x^0_i,\ub_i,\mu_i) \notag\\
& = C_i A_i^t x_i^0  +\! \sum_{\tau=0}^{t-1} C_i A_i^{t-1-\tau}\!( B_i u_i(\tau)+ G_i \mu_i ). \label{eq:yimean}
\end{align}
It is clear that random variables $\xb_i\vert_{x_i^0,\ub_i}(t)$
and $\yb_i\vert_{x_i^0,\ub_i}(t)$ can be degenerate or not, depending on the rank of matrices
$\Reach_t (A_i,G_i)$, $C_i \Reach_t (A_i,G_i)$, and $\Psi_i$.
We set $\Reach (A_i,G_i)=\Reach_{n_i} (A_i,G_i)$ 
(by Cayley-Hamilton theorem, $\im(\Reach_t (A_i,G_i))=\im(\Reach (A_i,G_i))$ for any time $t\geq n_i$).
Formulas \eqref{eq:xlin2} and \eqref{eq:ylin2} allow to compute the conditional covariances
\begin{align}
&  \cov\big(\xb_i\vert_{x_i^0,\ub_i}(t),\xb_i\vert_{x_i^0,\ub_i}(\tau)\big)
       =\sum_{h=0}^{\tau-1} A_i^{t-\tau+h} G_iG_i^T (A_i^{h})^T,   \label{eq:covxprocess}\\
&  \cov\big(\yb_i\vert_{x_i^0,\ub_i}(t),\yb_i\vert_{x_i^0,\ub_i}(\tau)\big) = \notag\\
& \hspace{2cm} \Psi_i + \sum_{h=0}^{\tau-1} C_iA_i^{t-\tau+h} G_iG_i^T (A_i^{h})^T C_i^T. \label{eq:covyprocess}
\end{align}
where $t\ge \tau$.
We also need to recall the notion of linear equivalence of stochastic linear control systems:

\begin{definition}   \label{def:LinEquiv}
Two stochastic linear control systems $\Sigma_1$ and $\Sigma_2$ as in \eqref{eq:systemSi} are linearly equivalent, denoted 
\[
\Sigma_1 \cong_{\lin} \Sigma_2,
\]
if $\Psi_1=\Psi_2$, $n_1=n_2$, and there exists an invertible matrix $\mathbb{T}\in\Real^{n_1 \times n_1}$, called transformation matrix, such that:
\begin{equation}
\label{algequiv}
\begin{array}
{lll}
A_2=\mathbb{T}  A_1 \mathbb{T} ^{-1} , & B_2=\mathbb{T}  B_1 ,             & C_2 = C_1 \mathbb{T} ^{-1},\\
G_2=\mathbb{T}  G_1 ,        & G_2 \mu_2=\mathbb{T}  G_1 \mu_1.  & 
\end{array}
\end{equation}
\end{definition}

The notion of linear equivalence is an equivalence relation on the class of linear systems.
We can now introduce the notion of equivalence of stochastic external behavior.

\begin{definition}    \label{def:ExtEquiv}
Consider two stochastic control systems $\Sigma_1$ and $\Sigma_2$ as in \eqref{eq:systemSi}, 
and a relation $\Rcal\subseteq \Real^{n_1}\times \Real^{n_2}$ that is a subspace.

$\Sigma_1$ and $\Sigma_2$ are said to have equivalent stochastic external behavior with respect to $\Rcal$ 
if for any $(x_1^0,x_2^0)\in\Rcal$ and for any input $\ub$
\begin{equation} \label{eq:eqprocy}
\yb_1\vert_{x_1^0,\ub} \sim \yb_2\vert_{x_2^0,\ub}.
\end{equation}
$\Sigma_1$ and $\Sigma_2$ are said to have equivalent stochastic external behavior, denoted 
\[
\Sigma_1 \cong_{\ext} \Sigma_2,
\]
if there exists a subspace total relation $\Rcal$ 
such that $\Sigma_1$ and $\Sigma_2$ have equivalent stochastic external behavior with respect to $\Rcal$.
\end{definition}

The above notion has been obtained by reinterpreting the notion of equivalence of external behavior given for behavioral systems, see e.g. \cite{BisimSchaft,Willems}, in a stochastic setting. The notion of equivalence of stochastic external behavior is an equivalence relation on the class of stochastic linear control systems.
We now proceed with a further step and propose a notion of stochastic bisimulation equivalence. 
We start by considering the case of linear systems with non-degenerate disturbances.

\begin{definition}
\label{defstochbisimnondeg}
Given a pair of stochastic control systems $\Sigma_1$ and $\Sigma_2$, as in \eqref{eq:systemSi} with $\rank(G_i) =n_i$, a subspace $\Rcal\subseteq \Real^{n_1} \times \Real^{n_2}$ is a stochastic bisimulation relation between $\Sigma_1$ and $\Sigma_2$ if for any pair $(x_1^0,x_2^0)\in\Rcal$ and any 
input $\ub$ the following conditions hold for all times $t\in \mathbb{N}$\\
(i) For any measurable set $X_1\subseteq \Rcal^{-1}(\Real^{n_2})$ \\
\[
\begin{array}
{l}
\mathbf{P}\!\left( \xb_1\vert_{x_1^0,\ub}(t) \in X_1\big|x_1^0\right)\!=
\mathbf{P}\!\left( \xb_2\vert_{x_2^0,\ub}(t) \in \Rcal(X_1)\big|x_2^0\right);
\end{array}
\]
(ii) For any measurable set $X_2\subseteq \Rcal(\Real^{n_1})$\\
\[
\begin{aligned}
 \mathbf{P}\!\left( \xb_2\vert_{x_2^0,\ub}(t)  \in X_2\big|x_2^0 \right) \!=
\mathbf{P}\!\left( \xb_1\vert_{x_1^0,\ub}(t) \in \Rcal^{-1}(X_2)\big|x_1^0\right);
\end{aligned}
\]
(iii) $\yb_1\vert_{x_1^0,\ub} \sim \yb_2\vert_{x_2^0,\ub} $.
\\[4pt]
Systems $\Sigma_1$ and $\Sigma_2$ are equivalent via stochastic bisimulation, if there exists a total stochastic bisimulation relation between them.
\end{definition}

Note that in conditions (i) and (ii) of Definition \ref{defstochbisimnondeg} we consider measurable sets $X_1\subseteq \Rcal^{-1}(\Real^{n_2})$ and $X_2\subseteq \Rcal(\Real^{n_1})$ rather than all measurable sets $X_1\subseteq \Real^{n_1}$ and $X_2\subseteq \Real^{n_2}$. 
  This choice is motivated by the fact that since relation $\Rcal$ may be not total, sets $\Rcal(X_1)$ and $\Rcal^{-1}(X_2)$ may be not defined for some sets $X_1\subseteq \Real^{n_1}$ and $X_2\subseteq \Real^{n_2}$ while they are defined for all sets $X_1\subseteq \Rcal^{-1}(\Real^{n_2})$ and $X_2\subseteq \Rcal(\Real^{n_1})$. 
	When $\Rcal$ is total, all measurable sets $X_1\subseteq \Real^{n_1}$ and $X_2\subseteq \Real^{n_2}$ are clearly considered. 
Definition \ref{defstochbisimnondeg} has been inspired by analogue notions given for probabilistic chains and Markov processes, see e.g. \cite{Larsen:91,HermannsBook,Desharnais2004}. A detailed discussion in this regard is reported in Section \ref{sec7}. 

\begin{remark}
As stressed in the introduction, this paper is within the research line (D2)--(R3)--(E1), where notions of stochastic bisimulation have been also proposed for General Stochastic Hybrid Systems (GSHS) with no inputs and outputs in \cite{BujorianuHSCC05} and, for Communicating Piecewise Deterministic Markov Processes (CPDMP) in \cite{StrubbeHSCC05,StrubbeHybridge}. A comparison of the proposed Definition \ref{defstochbisimnondeg} with the ones given in \cite{BujorianuHSCC05} and \cite{StrubbeHSCC05,StrubbeHybridge} follows. Although the mathematical tools employed in \cite{BujorianuHSCC05} are based on Category Theory and hence, different from the ones utilized in the present paper, the notion proposed in \cite{BujorianuHSCC05} in fact generalizes the one of stochastic bisimulation given for Labelled Markov Processes with countable sets of states in \cite{Blute:97} to GSHS. Since the definition given in \cite{Blute:97} generalizes the one given in \cite{Larsen:91} for probabilistic chains, then both Definition \ref{defstochbisimnondeg} and the one given in \cite{BujorianuHSCC05}, are in fact based on the same ideas given in the seminal work \cite{Larsen:91}. Regarding the comparison with the definitions of stochastic bisimulation given in \cite{StrubbeHSCC05,StrubbeHybridge} for CPDMP, we recall that the semantic of CPDMP is characterized by no stochasticity in the continuous--state flow; stochasticity only appears in the discrete--state dynamics, via spontaneous Poisson--type transitions and in the reset of both continuous and discrete variables. Since the systems in (\ref{eq:systemSi}) present stochasticity in the continuous--state flow and have no discrete--state dynamics, Definition \ref{defstochbisimnondeg} and the one given in \cite{StrubbeHSCC05,StrubbeHybridge} are not comparable. However, we mention that the definitions proposed in \cite{StrubbeHSCC05,StrubbeHybridge} are inspired by the one given in \cite{HermannsBook} 
for Interactive Markov Chains which combines the classical definition of bisimulation for concurrent (non--stochastic) processes (see e.g. \cite{Milner,Park}) with the one given in \cite{Larsen:91} for probabilistic chains. Hence, also in this case, the seminal work \cite{Larsen:91} is a common denominator in inspiring Definition \ref{defstochbisimnondeg} and the one given in \cite{StrubbeHSCC05,StrubbeHybridge}.
\end{remark}

As stressed at the beginning of this section, in this paper we consider linear systems with possibly degenerate disturbance distribution. 
The following example shows that Definition \ref{defstochbisimnondeg} is not appropriate to deal with linear systems with disturbances with degenerate distributions.

\begin{example}
\label{example0}
Consider a pair of stochastic control systems $\Sigma_1$ and $\Sigma_2$ as in (\ref{eq:systemSi}) where: 
\[
A_1=\begin{bmatrix} 
1 & 0\\
0 & 2
\end{bmatrix},
B_1=G_1=\begin{bmatrix} 1 \\ 0  \end{bmatrix},
C_1=\begin{bmatrix} 1 & 0 \end{bmatrix}
\]
$A_2=B_2=G_2=C_2=1$ and $w_i(t) \sim \mathcal{N}(0,1)$, without output noise $\nub_i$.
  The dynamics of $\Sigma_1$ and $\Sigma_2$ suggest that $\Sigma_1\cong_{\bis} \Sigma_2$ 
with stochastic bisimulation relation $\Rcal$ defined by $(x_1^0,x_2^0)\in \Rcal$ if and only if $x_1^0[1]=x_2^0$;
indeed, the dynamics of $x_2$ coincide with the dynamics of $x_1 [1]$, $y_2(t)=x_2(t)$ and $y_1(t) = x_1[1](t)$. 
We now apply Definition \ref{defstochbisimnondeg} only at time $t=1$. 
  We consider $(x_1^0,x_2^0)=(0,0)\in \Rcal$, $\ub
	= \zerob$ and 
the two measurable sets $X'_1, X''_1\subseteq \Real^{n_1}$ depicted in Fig. \ref{fig1}. 
We first note that $\Rcal(X'_1)=\Rcal(X''_1)=X'_2$, with $X'_2\subseteq \Real^{n_2}$ as depicted in Fig. \ref{fig1}. Hence, according to condition (i) of Definition \ref{defstochbisimnondeg}, a necessary condition for $\Sigma_1 \cong_{\bis} \Sigma_2$ is that 
\begin{equation}
\label{giobo}
\mathbf{P}(\xb_1|_{0,\zerob}(1) \in X'_1|0) =\mathbf{P}(\xb_1|_{0,\zerob}(1)\in X''_1|0).
\end{equation}
However, since $\supp(\xb_1|_{0,\zerob}(1))=\supp(G_1 w_1(0))=\im(G_1)$, 
by (\ref{supporto}) we get 
$\mathbf{P}(\xb_1|_{0,\zerob}(1) \in X'_1|0)=0$  and 
$\mathbf{P}(\xb_1|_{0,\zerob}(1) \in X''_1|0)\neq 0$, thus contradicting (\ref{giobo}). 
Hence, $\Sigma_1$ and $\Sigma_2$ are not equivalent via stochastic bisimulation according to Definition \ref{defstochbisimnondeg}. 
\end{example}

The above example motivates us to extend Definition \ref{defstochbisimnondeg} to linear systems with possibly degenerate disturbances, as follows:

\begin{definition}   \label{def:stochbisim}
Given a pair of stochastic control systems $\Sigma_1$ and $\Sigma_2$, as in (\ref{eq:systemSi}), a subspace $\Rcal\subseteq \Real^{n_1} \times \Real^{n_2}$ is a stochastic bisimulation relation between $\Sigma_1$ and $\Sigma_2$ if for any pair $(x_1^0,x_2^0)\in\Rcal$ and any input $\ub$ the following conditions hold for all times $t\in \mathbb{N}$\\
(i) For any measurable set $X_1\subseteq \Rcal^{-1}(\Real^{n_2})$ \\
\begin{equation}
\label{condI}
\begin{aligned}
& \mathbf{P}\!\left(\xb_1\vert_{x_1^0,\ub}(t)\in X_1\big|x_1^0\right)=\\
& \mathbf{P}\!\left(\xb_2\vert_{x_2^0,\ub}(t)\in 
   \Rcal\big(X_1\!\cap \supp(\xb_1\vert_{x_1^0,\ub}(t))\big)\big|x_2^0\right);
\end{aligned}
\end{equation}
(ii) For any measurable set $X_2\subseteq \Rcal(\Real^{n_1})$\\
\begin{equation}
\label{condII}
\begin{aligned}
& \mathbf{P}\!\left(\xb_2\vert_{x_2^0,\ub}(t)\in X_2\big|x_2^0\right)=\\
& \mathbf{P}\!\left(\xb_1\vert_{x_1^0,\ub}(t)\in 
   \Rcal^{-1}\big(X_2\!\cap \supp(\xb_2\vert_{x_2^0,\ub}(t))\big)\big|x_1^0\right);
\end{aligned}
\end{equation}
(iii) $\yb_1\vert_{x_1^0,\ub} \sim \yb_2\vert_{x_2^0,\ub}$.
\\[4pt]
Systems $\Sigma_1$ and $\Sigma_2$ are equivalent via stochastic bisimulation, denoted 
\[
\Sigma_1 \cong_{\bis} \Sigma_2,
\]
if there exists a total stochastic bisimulation relation between them.
\end{definition}

\vspace{5pt}

Note that by property (iii), if $\Sigma_1$ and $\Sigma_2$ are equivalent via stochastic bisimulation 
then they have equivalent stochastic external behavior. In the sequel, if not stated explicitly, when referring to equivalence via stochastic bisimulation we consider Definition \ref{def:stochbisim}.

\textit{Example \ref{example0}:} (Continued) 
When conditions (i) and (ii) of Definition \ref{defstochbisimnondeg} are replaced by conditions (i) and (ii) of Definition \ref{def:stochbisim}, one gets 
$\Rcal(X'_1\cap \supp(x_1(1)))=\Rcal(\varnothing)= \varnothing$  
and $\Rcal(X''_1\cap \supp(x_1(1)))=X''_2\neq\varnothing$. 
Therefore, condition (i) of Definition \ref{def:stochbisim} correctly distinguishes between sets $X'_1$ and $X''_1$, whereas condition (i) of Definition \ref{defstochbisimnondeg} does not. 
    A straightforward computation reveals indeed that systems $\Sigma_1$ and $\Sigma_2$ are equivalent via stochastic bisimulation according to Definition \ref{def:stochbisim}, while we showed they are not according to Definition \ref{defstochbisimnondeg}.

\begin{figure}
\begin{center}
\includegraphics[scale=0.45]{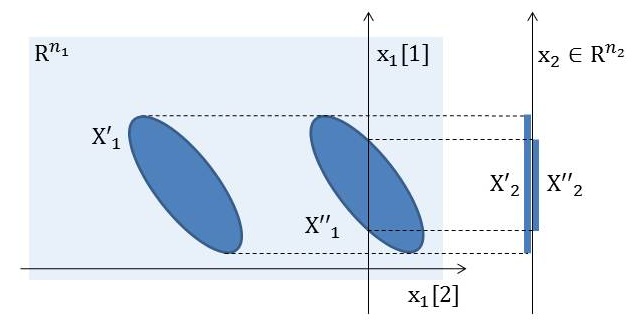} 
\caption{Illustration of the sets involved in Example \ref{example0}.}
\label{fig1}
\end{center}
\end{figure}
The notion of equivalence via stochastic bisimulation is an equivalence relation on the class of linear systems. 
Classical notions of bisimulation equivalences given for deterministic, non--deterministic and stochastic discrete (concurrent) processes preserve reachability properties of equivalent states, together with e.g. linear temporal logic properties, see e.g. \cite{ModelChecking,HermannsBook}. These notions only involve next states of equivalent states, rather than all states involved in runs originating from equivalent states. Definition \ref{def:stochbisim} clearly preserves reachability properties of states related by relation $\Rcal$. 
In contrast to the case of discrete processes, however, it requires properties (i)--(iii) to hold for all times $t\in\mathbb{N}$, rather than only for time $t=1$. 
We now show by a simple example that if the notion of stochastic bisimulation equivalence 
is defined in one step (i.e., only for $t=1$), 
stochastic reachability properties of states related by $\Rcal$ may be not preserved.

\begin{example}
Consider a pair of stochastic control systems $\Sigma_1$ and $\Sigma_2$ as in (\ref{eq:systemSi}) where: 
\[
\begin{array} {l}
A_1= \begin{bmatrix}  b & 1 \\  0 & a \end{bmatrix},
B_1 = 0_{2 \times 1},  
G_1=\begin{bmatrix} 0 \\ \sigma  \end{bmatrix},
C_1=\begin{bmatrix} 0 & 1  \end{bmatrix},  \\[8pt]
A_2= \begin{bmatrix}  b & 0 \\  0 & a \end{bmatrix},
B_2 = 0_{2 \times 1},  
G_2=\begin{bmatrix} 0 \\ \sigma  \end{bmatrix},
C_2=\begin{bmatrix} 0 & 1  \end{bmatrix},\\[8pt]
\end{array}
\]
without output noise $\nub_i$, with $a,b,\sigma \neq 0$. 
We first show that properties (i)--(iii) of Definition \ref{def:stochbisim} are satisfied for time $t=1$, that is to say that systems $\Sigma_1$ and $\Sigma_2$ are equivalent via stochastic bisimulation "in one step", as typically required for discrete processes. Consider the relation $\Rcal$ defined by $(x_1^0,x_2^0)\in\Rcal$ if and only if $x_1^0[2]=x_2^0[2]$
(recall that the argument within square brackets denotes the component of a vector).
We start with condition (i). 
Consider any $(x_1^0,x_2^0)\in\Rcal$, any input $\ub$, any measurable set $X_1\subseteq \Real^{n_1}$ and define 
$X_{1,2} =\{x_1[2]\in\Real | \exists x_1[1]\in\Real \text{ s.t. } 
x_1\in X_1 \cap\supp(\xb_1\vert_{x_1^0,\ub}(1))\}$. 
Then, one gets:
\[
\begin{aligned}
\mathbf{P}\big( & \xb_1\vert_{x_1^0,\ub}(1)) \in X_1|x_1^0\big) = 
\mathbf{P}\big(a x_1^0[2] +\sigma w_1(0) \in X_{1,2}\big)\\
     & = \mathbf{P}\big(a x_2^0[2] +\sigma w_2(0) \in X_{1,2}\big)\\
     & = \mathbf{P}\big(\xb_2\vert_{x_2^0,\ub}(1) \in \Rcal(X_{1}\cap \supp(\xb_1\vert_{x_1^0,\ub}(1)))|x_2^0\big),
\end{aligned}
\]
which is condition (i) of Definition \ref{def:stochbisim} for $t=1$.
Condition (ii) can be shown similarly and condition (iii) is trivially satisfied
because $C_i\xb_i\vert_{x_i^0,\ub_i}(1)=a x_i^0[2] +\sigma w_i(0)$ in this example.
We now show that $\Rcal$ does not satisfy those conditions for time $t=2$.
We consider $(x_i^0,\ub_i)=(0,\zerob)$ and obtain:
\[
\begin{aligned}
\xb_1\vert_{0,\zerob}(2) & = 
\begin{bmatrix}    0 & \sigma  \\
							\sigma & a\sigma \end{bmatrix}
\begin{bmatrix} \wb_1(1) \\ \wb_1(0) \end{bmatrix} \\[8pt]
\xb_2\vert_{0,\zerob}(2) & = 
\begin{bmatrix}    0 & 0\\
						\sigma & a\sigma \end{bmatrix}
\begin{bmatrix} \wb_2(1) \\ \wb_2(0) \end{bmatrix}
\end{aligned}
\]
Since $\xb_1\vert_{0,\zerob}(2)$ is non--degenerate, there exists a pair of sets in form of 
$X_1 =Z_1 \times Z_2$ and $X_1' =Z'_1 \times Z_2$ with $Z_1 \subset Z'_1$ such that:
\begin{equation} \label{pyt}
\mathbf{P}(\xb_1\vert_{0,\zerob}(2) \in X_1|0) < 
\mathbf{P}(\xb_1\vert_{0,\zerob}(2) \in X_1'|0).
\end{equation}
Considering that  $\supp(\xb_1\vert_{0,\zerob}(2))=\Real^2$,
condition (i) of Definition \ref{def:stochbisim} for set $X_1$ rewrites as:
\[
\mathbf{P}( \xb_1\vert_{0,\zerob}\!(2) \in Z_1 \times Z_2|0)=
\mathbf{P}( \xb_2\vert_{0,\zerob}\!(2) \in \Rcal(Z_1 \times Z_2)|0)
\]
Note that $\Rcal(Z_1 \times Z_2)=\Real \times Z_2$, because
$(x_1,x_2)\in\Rcal$ if $x_1[2]=x_2[2]$.
Thus
\[
\mathbf{P}( \xb_1\vert_{0,\zerob}(2) \in Z_1 \times Z_2|0)=
\mathbf{P}( \xb_2\vert_{0,\zerob}(2) \in \Real \times Z_2 |0)
\]
Analogously, condition (i) of Definition \ref{def:stochbisim} for set 
$X_1'=Z_1' \times Z_2$ rewrites as:
\[
\mathbf{P}( \xb_1\vert_{0,\zerob}(2) \in Z_1' \times Z_2|0)=
\mathbf{P}( \xb_2\vert_{0,\zerob}(2) \in \Real \times Z_2 |0)
\]
From these it follows
\[
\mathbf{P}( \xb_1\vert_{0,\zerob}(2) \in Z_1 \times Z_2|0)=
\mathbf{P}( \xb_1\vert_{0,\zerob}(2) \in Z_1' \times Z_2|0)
\]
which contradicts the inequality in (\ref{pyt}). 
Hence, condition (i) of Definition \ref{def:stochbisim} is not satisfied for time $t=2$,
although it is satisfied at $t=1$.
\end{example}

The above example motivated us to propose a definition of stochastic bisimulation equivalence in $t$ steps rather than in one step, as instead commonly done for discrete processes. \\
Connections between the notions introduced are now discussed. By comparing Definitions \ref{def:LinEquiv}, \ref{def:ExtEquiv} and \ref{def:stochbisim}, it is readily seen that:

\begin{proposition}   \label{prop:connections}
\textit{ }
\begin{itemize}
\item If $\Sigma_1 \cong_{\lin} \Sigma_2$ then $\Sigma_1 \cong_{\bis} \Sigma_2$;
\item If $\Sigma_1 \cong_{\bis} \Sigma_2$ then $\Sigma_1 \cong_{\ext} \Sigma_2$.
\end{itemize}
\end{proposition}

The converse implication of the first statement is not true in general, see e.g. Example \ref{example0} (Continued).
The converse implication of the second statement is also not true in general, as shown in the following example. 

\begin{example}
\label{exampleimpl}
Consider a pair of stochastic control systems $\Sigma_1$ and $\Sigma_2$ as in (\ref{eq:systemSi}) with  
\[
A_1=\begin{bmatrix} 1 & 0 \\
										0 & 2  \end{bmatrix},
B_1=\begin{bmatrix} 1 \\ 0 \end{bmatrix}, \quad 
\begin{aligned} 
G_1 & = I_2,\\
C_1 & =\begin{bmatrix} 1 & 0 \end{bmatrix},
\end{aligned}
\]
$A_2=B_2=G_2=C_2=1$, $w_1(t) \sim \mathcal{N}(0_2,I_2)$, $w_2(t) \sim \mathcal{N}(0,1)$, 
and without output noises $\nub_1$, $\nub_2$.
We have
\[
\begin{aligned}
\yb_1(t,x^0_1,\ub_1,\wb_1) & = x_1^0[1] + \sum_{\tau=0}^{t-1} (u_1(\tau)+w_1[1](\tau))\\
\yb_2(t,x^0_2,\ub_2,\wb_2) & = x_2^0 + \sum_{\tau=0}^{t-1} (u_2(\tau)+w_2(\tau)).
\end{aligned}
\]
Define the total relation $\Rcal\subseteq \Real^2 \times \Real$ by $(x_1,x_2)\in\Rcal$ if and only if $x_1[1]=x_2$. 
Then, for any $(x_1^0,x_2^0)\in\Rcal$ and input $\ub$:
\[
\yb_1|_{x^0_1,\ub} \sim \yb_2|_{x^0_2,\ub}
\]
and therefore $\Sigma_1$ and $\Sigma_2$ have equivalent stochastic external behavior ($\Sigma_1 \cong_{\ext} \Sigma_2$).
  We now show that $\Sigma_1$ and $\Sigma_2$ are not equivalent via stochastic bisimulation. 
	Suppose by contradiction that a total stochastic bisimulation relation $\Rcal$ exists between $\Sigma_1$ and $\Sigma_2$. 
Since $\Rcal$ is a subspace, it is always possible to find an invertible matrix $T\in\Real^{2\times 2}$ such that $(x_1,x_2)\in\Rcal$ if and only if $x_2=z_1 [1]$ where $z_1=T x_1$. 
   Consider $(x_1^0,x_2^0)=(0,0)\in\Rcal$ and select $\ub
	= \mathbf{0}$. 
   Consider the sets $X_{1,1}=T^{-1}([0,1]\times [0,1])$ and $X_{1,2}=T^{-1}([0,1]\times [0,2])$. 
	 Note that by construction $X_{1,1}\subset X_{1,2}$ and $\Rcal(X_{1,1})=\Rcal(X_{1,2})$.
Moreover, since $\supp\big(\xb_1|_{x_1^0,\ub_1}(1)\big)=\im(G_1)=\Real^2$ we have also:
\begin{equation}
\label{Rexample}
  \Rcal\!\big(X_{1,1}\cap \supp(\xb_1|_{x_1^0,\ub_1}(1))\big)
= \Rcal\!\big(X_{1,2}\cap \supp(\xb_1|_{x_1^0,\ub_1}(1))\big).
\end{equation}
From this, it easily follows that in order for condition (i) in Definition \ref{def:stochbisim}
to be satisfied it is necessary that
\begin{equation} \label{eq:condiproof} 
\mathbf{P}(\xb_1|_{x_1^0,\ub_1}(1) \in X_{1,1}|x_1^0)  
  = \mathbf{P}(\xb_1|_{x_1^0,\ub_1}(1) \in X_{1,2}|x_1^0).
\end{equation}
However, we can show that with the given choice of $X_{1,1}$ and $X_{1,2}$,
and $(x_1^0,x_2^0)=(0,0)$ and $\ub
= \mathbf{0}$, we have
\begin{equation} \label{eq:ineqcondi} 
\mathbf{P}(\xb_1|_{0,\zerob}(1) \in X_{1,1}|0)  
 < \mathbf{P}(\xb_1|_{0,\zerob}(1) \in X_{1,2}|0).
\end{equation}
thus contradicting \eqref{eq:condiproof}, and hence condition (i).
To prove inequality \eqref{eq:ineqcondi} note that for $x_1^0=0$ and $\ub=\zerob$
we have $\xb_1|_{0,\zerob}(1)=w_1(0)$.
Defining the nondegenerate random vector $v_1=T w_1(0)$, we get
\[
\begin{aligned}
\mathbf{P}(\xb_1|_{0,\zerob}(1) & \in X_{1,1}|x_1^0) 
    = \mathbf{P}(w_1(0)\in X_{1,1})\\
  = \ & \mathbf{P}(T w_1(0)\in T X_{1,1})\\
  = \ & \mathbf{P}(v_1\in [0,1]\times [0,1])\\
      & < \mathbf{P}(v_1\in [0,1]\times [0,2])\\
  & \quad = \mathbf{P}(T w_1(0) \in T X_{1,2})\\
  & \quad = \mathbf{P}(w_1(0)\in X_{1,2})\\
  & \quad = \mathbf{P}(\xb_1|_{0,\zerob}(1) \in X_{1,2}|x_1^0).
\end{aligned}
\] 
Thus, condition (i) in Definition \ref{def:stochbisim} cannot be satisfied for any total relation $\Rcal$,
and $\Sigma_1$ and $\Sigma_2$ are not equivalent via stochastic bisimulation.
\end{example}

\section{Geometric conditions}   \label{sec4}

In this section we derive geometric conditions characterizing the equivalence notions in Definitions \ref{def:ExtEquiv} and 
\ref{def:stochbisim}. 
Without loss of generality, we consider subspaces $\Rcal$ in Definitions \ref{def:ExtEquiv} and \ref{def:stochbisim} in the form of 
\begin{equation}  \label{eq:Rlin}
\Rcal=\ker (\begin{bmatrix} R_1 & -R_2 \end{bmatrix})\subseteq \Real^{n_1} \times \Real^{n_2},
\end{equation}
where $R_i\in\Real^{r\times n_i}$, $i=1,2$, so that
\begin{equation}
(x_1,x_2)\in \Rcal \quad \Longleftrightarrow\quad R_1 x_1 = R_2 x_2.
\end{equation}

\subsection{Some Technical Lemmas}

This subsection collects some technical results needed to prove the main results of the section. 
For any subsets $X_1\subseteq\Real^{n_1}$ and $X_2\subseteq\Real^{n_2}$ we have
\begin{equation} \label{eq:calcinvXi}
\Rcal(X_1) = R_2^{-1}(R_1 X_1),\ \ 
\Rcal^{-1}(X_2) = R_1^{-1}(R_2 X_2),
\end{equation}
as it easily follows from the identities below
\begin{equation} 
\begin{aligned}
\Rcal(X_1)& = \{x_2\in\Real^{n_2}\vert\, \exists x_1\in X_1     \ \text{s.t.} \  R_1 x_1 = R_2 x_2 \}\\
          & = \{x_2\in\Real^{n_2}\vert\, \exists   y\in R_1 X_1 \ \text{s.t.} \  y = R_2 x_2 \} \\
					& = R_2^{-1}(R_1 X_1),
\end{aligned}
\end{equation}
and similarly for $\Rcal^{-1}(X_2)$. Note that for some nonempty $X_1\subseteq\Real^{n_1}$ (or $X_2\subseteq\Real^{n_2}$) the set $\Rcal(X_1)$ (or $\Rcal^{-1}(X_2)$) can be empty, unless $\Rcal$ is a total relation.
	
\begin{proposition} \label{prop:total}
Relation $\Rcal$ as in (\ref{eq:Rlin}) is total if and only if $\im(R_1)=\im(R_2)$ or, equivalently, 
\begin{equation} \label{kost}
\rank(R_1)=\rank(R_2)=\rank(\begin{bmatrix} R_1 & -R_2 \end{bmatrix}).
\end{equation}
\end{proposition}

\begin{lemma}
\label{lemmaRlin}
Let $\Rcal$ be a relation as in (\ref{eq:Rlin}) and consider any sets $X_i\subseteq \Real^{n_i}, i=1,2$. 
Then:\\
(i) If $\Rcal$ is total then $R_1 X_1= R_2 \Rcal(X_1)$;\\
(ii) If $\Rcal$ is total then $R_1 \Rcal^{-1}(X_2)= R_2 X_2$;\\
(iii) $\Rcal(X_1)=\Rcal(X_1)+\ker(R_2)$;\\
(iv) $\Rcal^{-1}(X_2)=\Rcal^{-1}(X_2)+\ker(R_1)$.\\
\end{lemma}
\begin{proof}
{\it Proof of (i).} We first show $R_1 X_1 \subseteq R_2 \Rcal(X_1)$. Consider any $y\in R_1 X_1$. Then, there exists $x_1\in X_1$ such that $y=R_1 x_1$. Since $\Rcal$ is total then there exists $x_2\in\Real^{n_2}$ such that $(x_1,x_2)\in\Rcal$. Note that $x_2\in \Rcal(X_1)$. Since $(x_1,x_2)\in\Rcal$ then, by (\ref{eq:Rlin}), $R_1 x_1 = R_2 x_2$ which in turn, implies $y=R_2 x_2$. Since $R_2 x_2 \in R_2 \Rcal(X_1)$ we get $y\in R_2 \Rcal(X_1)$. We now show $R_2 \Rcal(X_1) \subseteq  R_1 X_1 $. 
Pick any $y \in R_2 \Rcal(X_1)$. Then there exists $x_2 \in \Rcal(X_1)$ such that $y=R_2 x_2$. Since $x_2 \in \Rcal(X_1)$ there exists $x_1 \in X_1$ such that $(x_1,x_2)\in \Rcal$ which implies $y=R_2 x_2 = R_1 x_1\in R_1 X_1$.
\\
{\it Proof of (ii).} This proof follows the same steps as those in the proof of (i).\\
{\it Proof of (iii).}
Since 
$\Rcal(X_1) \subseteq \Rcal(X_1)+\ker(R_2)$, we only need to show $\Rcal(X_1)+\ker(R_2) \subseteq \Rcal(X_1)$ or equivalently:
\begin{equation}
\label{ppm1}
x_2+z_2 \subseteq \Rcal(X_1), \forall x_2 \in \Rcal(X_1), \forall z_2 \in \ker(R_2).
\end{equation}
Consider any $x_2 \in \Rcal(X_1)$ and $z_2\in\ker(R_2)$. There exists $x_1 \in X_1$ such that $(x_1,x_2)\in\Rcal$ or equivalently, $R_1 x_1 = R_2 x_2$. Since $R_1 x_1 = R_2 x_2$ and $R_2 z_2=0$ then $R_1 x_1 = R_2 x_2+R_2 z_2=R_2 (x_2 + z_2)$, or equivalently $(x_1,x_2 +z_2)\in\Rcal$ from which, $x_2 +z_2 \in \Rcal(X_1)$. \\
{\it Proof of (iv).} This proof follows the same steps as those in the proof of (iii).
\end{proof}

For later purposes, we need to point out that 
$R_i\xb_i|_{x_1^0,\ub_1}$, for $i=1,2$, are Gaussian processes
whose conditional means at time $t\in\mathbb{N}$ are:
\begin{equation} \label{eq:condmeanRix}
\begin{aligned}
& E\left(R_i\xb_i|_{x_i^0,\ub_i}(t)\big\vert x_i^0\right) =  R_i \xb_i(t,x^0_i,\ub_i,\mub_i) \\
& \qquad = R_i A_i^t x_i^0 + \sum_{h=0}^{t-1} R_i A_i^{t-1-h}( B_i u_i(h) + G_i \mu_i )
\end{aligned}
\end{equation}
and conditional covariances, for $t\ge \tau$,
\begin{equation} \label{eq:condcovRix}
\begin{aligned}
&  \cov\big(R_i\xb_i\vert_{x_i^0,\ub_i}(t),R_i\xb_i\vert_{x_i^0,\ub_i}(\tau)\big) \\
& \hspace{2.8cm}  =\sum_{h=0}^{\tau-1} R_i A_i^{t-\tau+h} G_iG_i^T (A_i^{h})^T R_i^T.
\end{aligned}
\end{equation}

\begin{lemma}  \label{lem:R1R2processes}
Let $\Rcal=\ker([ R_1 \ -R_2])\subset \Real^{n_1+n_2}$ be a total stochastic bisimulation relation 
for two stochastic control systems $\Sigma_1$ and $\Sigma_2$ as in (\ref{eq:systemSi}).
Then, $\forall (x_1^0,x_2^0)\in\Rcal,$ and for any input $\ub$ we have
\begin{equation} \label{eq:implR1R2}
R_1 \xb_1\big\vert_{x_1^0,\ub} (t)\sim R_2 \xb_2\big\vert_{x_2^0,\ub} (t), \quad \forall t\ge 0.
\end{equation}
\end{lemma}

\begin{proof}
Let $r$ denote the number of rows of matrices $R_i$.
Condition \eqref{eq:implR1R2} is equivalent to claim that for any 
measurable $S\subset \Real^r$ it is 
$P(R_1 \xb_1\vert_{x_1^0,\ub} (t)\in S)=P(R_2 \xb_2\vert_{x_2^0,\ub} (t)\in S)$ for {\it any} 
$(x_1^0,x_2^0)\in\Rcal,$ input $\ub$, and $t\ge 0$.
The necessity of \eqref{eq:implR1R2} is proven by contradiction,
by showing that if $\Rcal$ is a stochastic bisimulation relation,
if for {\it some} $(x_1^0,x_2^0)\in\Rcal,$ input $\ub$, and $t\ge 0$,
there exists a subset $Q\subset \Real^r$ such that
$P(R_1 \xb_1\vert_{x_1^0,\ub} (t)\in Q)\not =P(R_2 \xb_2\vert_{x_2^0,\ub} (t)\in Q)$,
then we arrive at a contradiction.
	In particular, using the shorthand notation $v_i=\xb_i\vert_{x_i^0,\ub}$, $i=1,2$,
we will consider the case
\begin{equation} \label{eq:ineqqP}
P\big( R_1 v_1 \in Q\big) > P\big( R_2 v_2\in Q\big)
\end{equation}
(the case where the inequality is reversed 
can be handled with a symmetric reasoning).
Of course $P\big( R_i v_i \in Q\big)=P\big(v_i \in R_i^{-1}Q \big)$.
Then, by property (i) in Definition \ref{def:stochbisim} of stochastic bisimulation relation we have
\begin{equation} \label{eq:psssA}
\begin{aligned}
P\big( R_1 v_1 \in Q\big)  & = P\big( v_1 \in R_1^{-1} Q\big)  \\
		& = P\big( v_2 \in \Rcal((R_1^{-1} Q)\cap\supp(v_1))\big) 
\end{aligned}
\end{equation}
Obviously 
$\Rcal\big((R_1^{-1} Q)\cap\supp(v_1)\big)\subseteq \Rcal(R_1^{-1} Q)$, and by
eq.\ \eqref{eq:calcinvXi} we have
$\Rcal(R_1^{-1} Q)=R_2^{-1}\big(R_1(R_1^{-1} Q)\big) = R_2^{-1} Q$ 
then
\begin{equation} \label{eq:psssB}
P\big( v_2 \in \Rcal((R_1^{-1} Q)\cap\supp(v_1))\big) 
\le P\big( v_2 \in R_2^{-1} Q\big).
\end{equation}
Form \eqref{eq:psssA}, \eqref{eq:psssB} and
$P\big( v_2 \in R_2^{-1} Q\big) = P\big( R_2 v_2 \in Q\big)$, we get
\begin{equation} 
P\big( R_1 v_1 \in Q\big)  \le P\big( R_2 v_2 \in Q\big),
\end{equation}
contradicting the assumption $P\big( R_1 v_1 \in Q\big) > P\big( R_2 v_2 \in Q\big)$.
This proves the Lemma.
\end{proof}

Recalling eq.\ \eqref{eq:xlin2}, it is useful to define the following zero mean processes, for $i=1,2$,
\begin{equation} \label{eq:defxi}
\begin{aligned}
\xi_i(t) & =  \xb_i\vert_{x_i^0,\ub_i}(t) - \xb_i(t,x^0_i,\ub_i,\mub_i)\\
         & = \Reach_t (A_i,G_i)\wbtilde_i\vert_{0:t-1}
\end{aligned}
\end{equation}
which has the same covariances of $\xb_1|_{x_1^0,\ub_1}$ reported in \eqref{eq:covxprocess}
i.e., for $t\ge \tau$
\begin{equation}
\cov\big(\xi_i(t),\xi_i(\tau)\big)
       =\sum_{h=0}^{\tau-1} A_i^{t-\tau+h} G_iG_i^T (A_i^{h})^T.   \label{eq:covxiprocess}
\end{equation}
We also have, for any $x_i^0\in\Real^{n_1}$, $\ub$ and $t\ge0$,
\begin{equation} \label{eq:supporti}
\begin{aligned}
\supp(\xi_i(t)) & =\im(\Reach_t (A_i,G_i )),\\
\supp\big(\xb_i|_{x^0_i,\ub_i}(t)\big) & =\xb_i(t,x^0_i,\ub_i,\mub_i)+\im(\Reach_t (A_i,G_i )).  
\end{aligned}
\end{equation}

\begin{lemma} \label{prop:stochbisim}
Consider a stochastic linear control system $\Sigma_i$ as in (\ref{eq:systemSi}) and a matrix $R_i \in \Real^{r\times n_i}$. 
Let $A_i,G_i,R_i$ be such that
\begin{equation}
\label{condgio0}
\im(\Reach(A_i,G_i )) \cap  \ker(R_i)=\{0\}. 
\end{equation}
Then, 
\[
\mathbf{P}(R_i \xb_i|_{x^0_i,\ub_i}(t)\in R_i \overline{X}_i|x^0_i)=
\mathbf{P}(    \xb_i|_{x^0_i,\ub_i}(t)\in     \overline{X}_i|x^0_i),
\] 
for any time $t \in \mathbb{N}$, initial condition $x^0_i \in\Real^{n_i}$, control input sequence $\ub_i$, 
and any measurable set 
\begin{equation}
\label{condLemma}
\overline{X}_i \subseteq \supp\big(\xb_i|_{x^0_i,\ub_i}(t)\big).
\end{equation}
\end{lemma}

\begin{proof}
Pick any set $\overline{X}_i \subset\Real^{n_i}$ satisfying (\ref{condLemma}).
and define the shifted set $H_i = \overline{X}_i - \xb_i(t,x^0_i,\ub_i,\mub_i)$.
By \eqref{eq:supporti}  $H_i \subseteq \im(\Reach_t (A_i,G_i )) \subseteq \im(\Reach(A_i,G_i ))$, so that 
$\xb_i\vert_{x_i^0,\ub_i}(t) \in \overline{X}_i$ if and only if $\xi_i(t)\in H_i$ and 
$\mathbf{P}(\xb_i\vert_{x_i^0,\ub_i}(t)\in \overline{X}_i| x_i^0) = \mathbf{P}(\xi_i(t)\in H_i)$. 
The thesis is proven if we show that $\mathbf{P}(R_i\xi_i(t)\in R_i H_i ) = \mathbf{P}(\xi_i(t)\in H_i)$. 
This is straightforward because $R_i\xi_i(t)\in R_i H_i$ implies $\xi_i(t)\in (H_i + \kr(R_i))$; by construction $\xi_i(t)\in 
\im(\Reach_t (A_i,G_i )) \subseteq \im(\Reach(A_i,G_i ))$ and therefore by (\ref{condgio0}), $\xi_i(t)\not\in \kr(R_i)\backslash \{0\}$. Thus, $\xi_i(t) \in (H_i + \kr(R_i))$ implies $\xi_i(t)\in H_i$, which concludes the proof.
\end{proof}

\subsection{Geometric conditions for stochastic external behavior equivalence}

This section collects some algebraic and geometric conditions that characterize stochastic external equivalence relations.

\begin{proposition}  \label{prop:necessequiv1}
If two stochastic systems $\Sigma_1$ and $\Sigma_2$, as in (\ref{eq:systemSi})
have equivalent stochastic external behavior with respect to some subspace relation $\Rcal\subseteq \Real^{n_1+n_2}$,
then 
\[
\hspace{-3pt}
\begin{aligned}
& p_0) & & \Psi_1=\Psi_2,\\
& p_1) & & C_1 A_1^t B_1 = C_2 A_2^t B_2, \qquad\quad \forall t\ge 0, \\
& p_2) & & C_1 A_1^t G_1\mu_1 = C_2 A_2^t G_2\mu_2, \qquad \forall t\ge 0, \\
& p_3) & & \hspace{-4pt} C_1 A_1^k G_1 G_1^T\! (A_1^h)^T \! C_1^T \!=\! 
                         C_2 A_2^k G_2 G_2^T\! (A_2^h)^T \! C_2^T, \, \forall h,k\ge 0.
\end{aligned}
\]
\end{proposition}

\begin{proof}
All the listed conditions are a straightforward consequence of the formula \eqref{eq:yimean} of expected values 
of the output processes $y_i\big|_{x_i^0,\ub_i}$ for any ($x_i^0,\ub_i$), used with $x_i^0=0$, 
and of the formula \eqref{eq:covyprocess} of covariances.
\end{proof}
In order to derive equivalent geometric conditions we find it useful to define the following extended system $\Sigmatilde$ of dimension $\ntilde=n_1+n_2$
\begin{equation}  \label{eq:extsys}
\begin{aligned}
\Atilde & \!=\! \diag(A_1,A_2)\!=\!\begin{bmatrix} A_1 & 0 \\ 0 & A_2 \end{bmatrix},\ 
						\Btilde=\col(B_1,B_2)=\begin{bmatrix} B_1 \\ B_2 \end{bmatrix},\\
\Ctilde& \!=\! \row(C_1,-C_2)=\begin{bmatrix} C_1 & -C_2 \end{bmatrix}.
\end{aligned}
\end{equation}
and the following matrices:
\[
Q_{i,k} = \Obs_{k}(A_i,C_i),\quad 
P_{i,k} = \Reach_{k}(A_i,B_i),\quad i=1,2.
\]
Now we can state the following geometric conditions:

\begin{proposition}  \label{prop:necessequiv2}
Consider two stochastic systems $\Sigma_1$ and $\Sigma_2$, as in (\ref{eq:systemSi}).
The following conditions:
\[
\begin{aligned}
& p_1')  & &  \im\begin{bmatrix} B_1 \\ B_2 \end{bmatrix} \subseteq
            \ker\begin{bmatrix} Q_{1,\ntilde} & -Q_{2,\ntilde} \end{bmatrix}\\
& p_2') & &  \im\begin{bmatrix} G_1\mu_1 \\ G_2\mu_2 \end{bmatrix} \subseteq
            \ker\begin{bmatrix} Q_{1,\ntilde} & -Q_{2,\ntilde} \end{bmatrix}\\
& p_3') & & \exists H\in \Real^{l_1\times l_2}: \ 
\im\begin{bmatrix} G_1 H \\ G_2 \end{bmatrix} \subseteq
            \ker\begin{bmatrix} Q_{1,\ntilde} & -Q_{2,\ntilde} \end{bmatrix}
\end{aligned}
\]
are necessary for the stochastic external behavior equivalence of systems $\Sigma_1$ and $\Sigma_2$
with respect to some relation $\Rcal$.
\end{proposition}

\begin{proof}
The proof is based on the observability analysis of the extended systems \eqref{eq:extsys},
and is obtained by showing that each condition ($p'_i$), $i=1,2,3$, is equivalent to the
corresponding condition ($p_i$) of Proposition \ref{prop:necessequiv1}.
By observing that:
\[
\begin{aligned}
\begin{bmatrix} Q_{1,\ntilde} & -Q_{2,\ntilde} \end{bmatrix}\! =\! \Obs_{\ntilde}(\Atilde,\Ctilde),
\end{aligned}
\]
condition $(p'_1)$ can be written as $\im(\Btilde)\subseteq  \Obs_{\ntilde}(\Atilde,\Ctilde)$,
and this implies that $\Ctilde \Atilde^t \Btilde =0$,$\forall t\ge0$, that is equivalent to $(p_1)$
of Proposition \ref{prop:necessequiv1} thanks to the block diagonal structure of system $\Sigmatilde$ 
\eqref{eq:extsys}.
A similar reasoning proves that $(p'_2)$ is equivalent to $(p_2)$.
 Equivalence of ($p_3'$) to ($p_3$) is proved by observing that ($p_3$) is equivalent
to 
\[
Q_{1,\ntilde} G_1 G_1^T Q_{1,\ntilde}^T = Q_{2,\ntilde} G_2 G_2^T Q_{2,\ntilde}^T,
\]
which is equivalent to $\im(Q_{1,\ntilde} G_1)=\im(Q_{2,\ntilde} G_2)$, and this in turn is equivalent
to $Q_{1,\ntilde} G_1 H = Q_{2,\ntilde} G_2$ for some matrix $H$.
From this, equivalence of ($p_3$) and ($p_3'$) easily follows.
\end{proof}

Note that the algebraic conditions $(p_0)$--$(p_3)$ of Proposition \ref{prop:necessequiv1}
and the geometric conditions $(p_1')$--$(p_3')$ of Proposition \ref{prop:necessequiv2}
do not depend on the choice of the relation $\Rcal$. The following results provide necessary and sufficient conditions on the subspace relation $\Rcal\subseteq\Real^{n_1+n_2}$
that ensure stochastic external equivalence of two systems $\Sigma_1$ and $\Sigma_2$.

\begin{theorem}   \label{thm:condRequiv}
Stochastic control systems $\Sigma_1$ and $\Sigma_2$ have equivalent stochastic external behavior with respect to $\Rcal$ 
if and only if they satisfy conditions $(p_0)$--$(p_3)$ of Proposition \ref{prop:necessequiv1}
(or the equivalent ones in Proposition \ref{prop:necessequiv2}) and $\Rcal$ satisfies the following one:
\begin{equation} \label{eq:condRincl} 
p_4)\qquad
\Rcal\subseteq \ker(\begin{bmatrix} Q_{1,\ntilde} & -Q_{2,\ntilde} \end{bmatrix}).\hspace{2cm}
\end{equation}
Moreover, there always exists a $\diag(A_1,A_2)$-invariant subspace relation $\Rcal'\supseteq\Rcal$
such that $\Sigma_1$ and $\Sigma_2$ have equivalent stochastic external behavior with respect to $\Rcal'$.
\end{theorem} 

\begin{proof}
The first assertion is proved by noting that condition \eqref{eq:condRincl} is equivalent to 
\begin{equation} 
C_1 A_2^t x_1^0 = C_2 A_2^t x_2^0, \quad \forall t\ge 0, \quad \forall (x_1^0,x_2^0)\in\Rcal,
\end{equation}
which is clearly necessary and sufficient, together with ($p_1$) and ($p_2$) of Proposition
\ref{prop:necessequiv1}, to guarantee that the processes $y_1|_{x_1^0,\ub_1}$ 
and $y_2|_{x_2^0,\ub_2}$ have the same expected values (see equation \eqref{eq:yimean}). 
The second assertion is easily proved by observing that it is trivially verified
by choosing $\Rcal'= \ker(\begin{bmatrix} Q_{1,\ntilde} & -Q_{2,\ntilde} \end{bmatrix})$,
which is clearly $\diag(A_1,A_2)$-invariant. 
\end{proof}

\begin{corollary} \label{cor:maximalRequiv}
If $\Sigma_1$ and $\Sigma_2$ satisfy all conditions $(p_0)$--$(p_3)$ of Proposition \ref{prop:necessequiv1},
the largest relation in $\Real^{n_1 + n_2}$ that ensures equivalent stochastic external behavior is 
\begin{equation}
\label{RmaxL}
\ker(\begin{bmatrix} Q_{1,\ntilde} & -Q_{2,\ntilde} \end{bmatrix}).
\end{equation}
\end{corollary}

\begin{remark} \label{rem:condstochextequiv}

Theorem \ref{thm:condRequiv} implies that if 
$\Sigma_1$ and $\Sigma_2$ have equivalent stochastic external behavior with respect to a relation
$\Rcal$, either $\Rcal$ is $\diag(A_1,A_2)$-invariant or is contained in a larger relation $\Rcal'$
which is $\diag(A_1,A_2)$-invariant.
Corollary \ref{cor:maximalRequiv} states that the maximal relation that ensures stochastic
external equivalence is $\ker(\Obs_\ntilde(\Atilde,\Ctilde))$.
\end{remark}

The following result then easily follows:

\begin{corollary}
\label{cor:Lequiv}
Linear systems $\Sigma_1$ and $\Sigma_2$ have equivalent stochastic external behavior if and only if conditions ($p_0$)--($p_3$) hold and relation (\ref{RmaxL}) is total.
\end{corollary}

\subsection{Geometric conditions for stochastic bisimulation equivalence}

The following result provides geometric conditions for characterizing stochastic bisimulation relations.

\begin{theorem} \label{thm:condRbisimil}  
Consider two stochastic systems $\Sigma_1$ and $\Sigma_2$, as in (\ref{eq:systemSi}),
such that $\Psi_1=\Psi_2$ and a subspace total relation $\Rcal$ as in \eqref{eq:Rlin} enjoying the following $(A_1,A_2)$--invariance condition

\begin{equation}   \label{eq:diagA1A2invar}
h_0)\quad \diag(A_1,A_2)\, \Rcal \subseteq \Rcal. \hspace{3cm}
\end{equation}

Then, $\Rcal$ is a stochastic bisimulation relation between $\Sigma_1$ and $\Sigma_2$ if and only if the following conditions are satisfied:
	
\begin{equation}   \label{eq:cond_h1-h4}
\begin{array} {l}
h_1)\ \im\big(\col(B_1,B_2) \big) \subseteq \Rcal  \text{ (i.e.}\ R_1 B_1 = R_2 B_2\text{)}, \\
h_2)\ \im\big(\col(G_1\mu_1,G_2\mu_2) \big) \subseteq \Rcal \text{ (i.e.}\ R_1 G_1\mu_1 \!=\! R_2 G_2\mu_2\text{)}, \\
h_3)\ R_1 G_1 G_1^T R_1^T = R_2 G_2 G_2^T R_2^T,  \\
h_4)\ \Rcal\subseteq \ker(\begin{bmatrix} C_1 & -C_2 \end{bmatrix}), \\
h_5) \ \im(\Reach(A_i,G_i )) \cap  \ker(R_i)=\{0\}, \ i=1,2. 
\end{array}
\end{equation}
\end{theorem}

\begin{proof}
(Sufficiency) 
First of all note that the properties ($h_0$), ($h_1$) and ($h_2$) imply the following 
\begin{equation}   \label{eq:QAS}
\begin{aligned}
& h_0')  &  R_1 A_1^t x_1^0    & = R_2 A_2^t x_2^0, & & \forall t\in\mathbb{N},\ \forall (x_1^0,x_2^0)\in\Rcal,  \\
& h_1')  &  R_1 A_1^t B_1      & = R_2 A_2^t B_2,   & & \forall t\in\mathbb{N}& \\
& h_2')  &  R_1 A_1^t G_1\mu_1 & = R_2 A_2^t G_2\mu_2. & & \forall t\in\mathbb{N}& \\    
\end{aligned}
\end{equation}
Property ($h_3$) is equivalent to the existence of a  matrix
$H$ such that $R_1 G_1 H = R_2 G_2 $. 
Then, by the invariance property ($h_0$) it follows that
$R_1 A_1^t G_1 H = R_2 A_2^t G_2 $, for any $t\in\mathbb{N}$ and therefore 
for all $k,h\in\mathbb{N}$
\begin{equation}
\label{eq:QAS1}
\hspace{-0.1cm}
h_3')\quad R_1 A_1^k G_1 G_1^T (A_1^h)^T R_1^T = R_2 A_2^k G_2 G_2^T (A_2^h)^T R_2^T.
\end{equation}
From $(h_0')$--$(h_3')$ it easily follows that the Gaussian processes
$R_1 \xb_1|_{x^0_1,\ub_1}$ and 
$R_2 \xb_2|_{x^0_2,\ub_1}$, when $(x^0_1,x^0_2)\in\Rcal$ and $\ub_1=\ub_2$,
have the same means \eqref{eq:condmeanRix}
and covariances \eqref{eq:condcovRix}, and therefore
\begin{equation} \label{eq:cond1}
R_1 \xb_1|_{x^0_1,\ub}\sim R_2 \xb_2|_{x^0_2,\ub},\quad \forall (x_1^0,x_2^0)\in\Rcal.
\end{equation}
For any measurable set $X_1\subseteq\Real^{n_1}$ and any $t\in\mathbb{N}$ define 
\begin{equation}
\label{sets}
\begin{array}
{l}
\overline{X}_1= X_1\cap \supp(\xb_1\vert_{x^0_1,\ub_1}(t)),\\
X_2=\Rcal(\overline{X}_1).\\
\end{array}
\end{equation}
By definition of sets $X_1$ and $\overline{X}_1$ and by (\ref{supporto}) we get for all $t\in\mathbb{N}$:
\begin{equation}   \label{cond51a}
\mathbf{P}\big(\xb_1\vert_{x^0_1,\ub_1}(t) \in  X_1 | x_1^0\big)= 
\mathbf{P}\big(\xb_1\vert_{x^0_1,\ub_1}(t) \in  \overline{X}_1 | x_1^0\big).
\end{equation}
Since by construction, set $\overline{X}_1$ satisfies condition (\ref{condLemma}), by condition ($h_5$) in the statement and Lemma \ref{prop:stochbisim}, we get for all $t\in\mathbb{N}$: 
\begin{equation}   \label{cond511}
\mathbf{P}\big(\xb_1\vert_{x^0_1,\ub_1}(t) \in \overline{X}_1 | x_1^0\big)=
\mathbf{P}\big(R_1 \xb_1\vert_{x^0_1,\ub_1}(t) \in R_1 \overline{X}_1 | x_1^0\big).
\end{equation}
Since $\Rcal$ is total, by Lemma \ref{lemmaRlin} (i), we get: 
\begin{equation}   \label{cond211}
R_1 \overline{X}_1 = R_2 X_2. 
\end{equation}
By combining \eqref{eq:cond1} and \eqref{cond211} we get for all $t\in\mathbb{N}$:
\begin{equation}  \label{cond3}
\mathbf{P}(R_1 \xb_1\vert_{x^0_1,\ub_1}(t) \in R_1 \overline{X}_1 | x^0_1)=
\mathbf{P}(R_2 \xb_2\vert_{x^0_2,\ub_2}(t) \in R_2 X_2 | x_2^0).
\end{equation}
By definition of set $X_2$ and by Lemma \ref{lemmaRlin} (iii), we get for all $t\in\mathbb{N}$:
\begin{equation} \label{cond61}
\begin{aligned}
\mathbf{P}\big(R_2 & \xb_2\vert_{x^0_2,\ub_2}(t) \in R_2 X_2 | x_2^0\big) \\
   & = \mathbf{P}\Big(\xb_2\vert_{x^0_2,\ub_2}(t) \in \big(X_2 +\ker(R_2)\big) \big| x_2^0\Big)\\
   & = \mathbf{P}\big(\xb_2\vert_{x^0_2,\ub_2}(t) \in X_2 | x_2^0\big).
\end{aligned}
\end{equation}
By combining the equalities in (\ref{cond51a}), (\ref{cond511}), (\ref{cond3}) and (\ref{cond61}), we get condition (i) of Definition \ref{def:stochbisim}. 
Condition (ii) of Definition \ref{def:stochbisim} can be shown by using a symmetric reasoning. 
In particular, while the proof of condition (i) makes use of 
Lemma \ref{prop:stochbisim}, Lemma \ref{lemmaRlin} (i) and Lemma \ref{lemmaRlin} (iii), the proof of condition (ii) makes use of Lemma \ref{prop:stochbisim}, Lemma \ref{lemmaRlin} (ii) and Lemma \ref{lemmaRlin} (iv).
Regarding condition (iii) of Definition \ref{def:stochbisim}, note that condition ($h_4$) 
implies existence of a matrix $S$ such that 
$S[R_1 \ -R_2]=[C_1 \ -C_2]$, i.e.\ $SR_1=C_1$ and $SR_2=C_2$.
  Thus, from condition \eqref{eq:cond1}, considering that $\Psi_1=\Psi_2$ implies $\nub_1\sim\nub_2$,
by the independence of $\xb_i$ and $\nub_i$ we have, for any $(x_1^0,x_2^0)\in\Rcal$ and input $\ub$,
\[
S R_1 \xb_1\vert_{x^0_1,\ub} + \nub_1 \sim S R_2 \xb_2\vert_{x^0_2,\ub} + \nub_2.
\]
Since $\yb_i\vert_{x^0_i,\ub_i} = S R_i \xb_1\vert_{x^0_i,\ub_i} + \nub_i$,
condition (iii) of Definition \ref{def:stochbisim} is satisfied.\\
(Necessity)
By Lemma \ref{lem:R1R2processes}, if $\Rcal$ is a stochastic bisimulation relation, then
for any $(x_1^0,x_2^0)\in\Rcal,$ and input $\ub$ we have
$R_1 \xb_1\big\vert_{x_1^0,\ub} (t)\sim R_2 \xb_2\big\vert_{x_2^0,\ub} (t)$, $\forall t\ge 0.$
Then, necessarily 
$E(R_1 \xb_1\big\vert_{x_1^0,\ub} (1)|x_1^0)=E(R_2 \xb_2\big\vert_{x_2^0,\ub} (1)|x_2^0)$,
that means $R_1A_1x_1^0 + R_1B_1u_1(0)+R_1G_1\mu_1 = R_2A_2x_2^0 + R_2B_2u_2(0)+R_2G_2\mu_2$ 
for any $(x_1^0,x_2^0)\in\Rcal,$ and input $\ub$. This easily implies conditions $(h_1)$ and $(h_2)$.
Moreover, equating the covariances of $R_1 \xb_1\big\vert_{x_1^0,\ub} (1)$ and
$R_2 \xb_2\big\vert_{x_2^0,\ub} (1)$ the condition $(h_3)$ follows.
Moreover, if $\Rcal$ is a stochastic bisimulation relation, then 
$\Sigma_1$ and $\Sigma_2$ have equivalent stochastic external behavior with respect to $\Rcal$
(Proposition \ref{prop:connections})
then necessarily the inclusion \eqref{eq:condRincl} holds true, which in turn implies condition $(h_4)$.
It remains to show that if $\Rcal$ is a total stochastic bisimulation relation between $\Sigma_1$ and $\Sigma_2$
then necessarily ($h_5$) is verified.
The proof is obtained by contradiction.
Let $\Rcal=\ker([R_1 \ -R_2])$ be a total stochastic bisimulation relation between $\Sigma_1$ and $\Sigma_2$. 
 We will prove that for any pair $(x_1^0,x_2^0)\in\Rcal$, i.e.\ such that $R_1x_1^0=R_2 x_2^0$, 
if $\im(\Reach(A_i,G_i )) \cap  \ker(R_i) \not = \{0\}$ for $i=1$ or $i=2$ then 
the properties (i) or (ii) of Definition \ref{def:stochbisim}
are not verified for some set $X_1$ such that $R_1X_1=R_2 \Rcal(X_1)$, 
or $X_2$ such that $R_2 X_2=R_1 \Rcal^{-1}(X_2)$, and hence $\Rcal$ is not a stochastic bisimulation relation. 
We only give the proof for $i=1$ since the case $i=2$ follows a symmetric reasoning. 
Note first that if $\im(\Reach(A_1,G_1 )) \cap \ker(R_1) \not = \{0\}$ then there exists $v\in\Real^{n_1}$, with $v\neq 0_{n_1}$ such that 
$v\in \im(\Reach_t (A_1,G_1))$ and $v\in\kr(R_1)$.
Consider now the random vector $\xi_1(t)$ defined in \eqref{eq:defxi} with $i=1$, and consider any time $t\ge n_1$.
We can always take a pair of sets $H_1\subset\im(\Reach (A_1,G_1))$ 
and $V_1\subseteq \im(v)$ (so that $V_1\subseteq \im(\Reach (A_1,G_1))\cap\kr(R_1)$) such that 
\begin{equation} \label{eq:HiVi}
\mathbf{P}(\xi_1(t)\in H_1) < \mathbf{P}\big(\xi_1(t)\in (H_1 + V_1)\big).
\end{equation}
By defining the set $X_1 = H_1 + \xb_1(t,x^0_1,\ub_1,\mub_1)$ we have
\begin{equation}
\begin{aligned}
\mathbf{P}\big(\xb_1\vert_{x^0_1,\ub_1}(t) \in X_1 | x_1^0\big) & = \mathbf{P}\big(\xi_1(t)\in H_1\big),\\
\mathbf{P}\big(\xb_1\vert_{x^0_1,\ub_1}(t) \in (X_1+ V_1)| x_1^0\big) & =
							\mathbf{P}\big(\xi_1(t)\in (H_1 + V_1)\big),
\end{aligned}
\end{equation}
so that from \eqref{eq:HiVi}
\begin{equation} \label{eq:XiVi_ineq}
\begin{aligned}
\mathbf{P}\big(\xb_1\vert_{x^0_1,\ub_1}(t) \in X_1 | x_1^0\big) 
< \mathbf{P}\big(\xb_1\vert_{x^0_1,\ub_1}(t) \in (X_1+ V_1)| x_1^0\big)
\end{aligned}
\end{equation}
Since both $H_1$ and $V_1$ belong to $\im(\Reach (A_1,G_1))$ we have 
\begin{equation} 
\begin{aligned}
X_1       & \subset \supp\big(\xb_1\vert_{x^0_1,\ub_1}(t)\big),\\
X_1 + V_1 & \subset \supp\big(\xb_1\vert_{x^0_1,\ub_1}(t)\big),
\end{aligned}
\end{equation}
so that
\begin{equation} \label{suppnec}
\begin{aligned}
\Rcal\big(X_1\cap\supp\big(\xb_1\vert_{x^0_1,\ub_1}(t)\big)\big) & = \Rcal(X_1),\\
\Rcal\big((X_1+V_1)\cap\supp\big(\xb_1\vert_{x^0_1,\ub_1}(t)\big)\big) &= \Rcal(X_1+V_1),
\end{aligned}
\end{equation}
By assumption the pair $(x_1^0,x_2^0)$ belongs to the total relation $\Rcal=\kr([R_1\ -R_2])$, so that 
by property (i) of Definition \ref{def:stochbisim} and by \eqref{suppnec} we have
\begin{equation} \label{gio1}
\begin{aligned}
\mathbf{P}\big( \xb_1\vert_{x^0_1,\ub_1} & (t)  \in X_1| x_1^0\big) \\
   & = \mathbf{P}\big( \xb_2\vert_{x^0_2,\ub_2}(t) \in \Rcal(X_1)| x_2^0 \big),\\
\mathbf{P}\big( \xb_1\vert_{x^0_1,\ub_1} & (t)  \in X_1 + V_1| x_1^0\big) \\
   & = \mathbf{P}\big( \xb_2\vert_{x^0_2,\ub_2}(t) \in \Rcal(X_1 + V_1)| x_2^0 \big),
\end{aligned}
\end{equation}
Since, by construction, $V_1\subseteq \kr(R_i)$, i.e.\ $R_1 V_1 = 0$, it follows that 
$R_1 X_1 = R_1 (X_1 + V_1)$, and from this $\Rcal(X_1)=\Rcal(X_1 + V_1)$.
 Thus, from \eqref{gio1} we get:
\begin{equation} 
\begin{aligned}
\mathbf{P}\big(\xb_1\vert_{x^0_1,\ub_1}(t)  \in X_1| x_1^0\big) =
\mathbf{P}\big(\xb_1\vert_{x^0_1,\ub_1}(t)  \in X_1 + V_1| x_1^0\big).
\end{aligned}
\end{equation}
which contradicts the inequality \eqref{eq:XiVi_ineq}.
Hence, condition ($h_5$) must necessarily be satisfied, and the Theorem is proved.
\end{proof}

\begin{remark}
Conditions ($h_1$)--($h_5$) in the above result are necessary and sufficient in the special case when relation $\mathcal{R}$ is $(A_1,A_2)$--invariant. Condition ($h_0$) is indeed not necessary. In fact, it is not difficult to construct a pair of systems $\Sigma_1$ and $\Sigma_2$ for which there exists a total relation $\mathcal{R}$ satisfying conditions ($h_1$)--($h_4$) and admitting a proper total relation subspace (which is then still a total stochastic bisimulation relation between $\Sigma_1$ and $\Sigma_2$) that is not $(A_1,A_2)$--invariant. 
However, as discussed in Section \ref{sec5}, when performing model reduction via stochastic bisimulation equivalence, 
stochastic bisimulation relations $\mathcal{R}$ involved satisfy indeed condition ($h_0$). Hence, in this respect, condition ($h_0$) is not limiting.
\end{remark}

We conclude this section with a specialization of Theorem \ref{thm:condRbisimil} to the non--degenerate case:

\begin{corollary}   \label{cor:nondeg}
Consider systems $\Sigma_1$ and $\Sigma_2$ as in \eqref{eq:systemSi}, 
and suppose that $\rank(\Reach(A_i,G_i))=n_i$, $i=1,2$ (i.e., the 
probability measure on both the state spaces is non-degenerate after $n_i$ steps).
	Then $\Sigma_1$ and $\Sigma_2$ are equivalent via stochastic bisimulation if and only if they are linearly equivalent.
\end{corollary}

\begin{proof}
The sufficiency comes from Proposition \ref{prop:connections}. 
As far as the necessity, since $\rank(\Reach(A_i,G_i))=n_i$ then condition ($h_5$) 
of Theorem \ref{thm:condRbisimil} boils down to $\rank(R_i)=n_i$. 
Then, necessarily $n_i\leq r$ where $r$ is the number of rows of $R_1$ and $R_2$. 
Moreover, by assumption $\rank(R_1)=\rank(R_2)$ because $\Rcal$ is total, and therefore necessarily
$\Sigma_1$ and $\Sigma_2$ have the same dimension ($n_1=n_2$).
   Hence, the result follows by defining the nonsingular transformation matrix 
$\mathbb{T} = R_2^{+} R_1$, where $R_2^+=(R_2^T R_2)^{-1} R_2^T $ is the Moore--Penrose pseudo inverse matrix of $R_2$,
and verifying that $(x_1,x_2)\in\Rcal\ \Leftrightarrow \ x_2=\mathbb{T}x_1$.
\end{proof}

\section{Model Reduction} \label{sec5}

In this section we consider a linear system 
\begin{equation}  \label{eq:System}
\Sigma:
\left\{
\begin{aligned}
x(t+1) & = A x(t) + B u(t) + G w(t),\\
  y(t) & = C x(t) + \nu(t),\\
x \in\Real^{n}, & \ u \in \Real^{m},\, w \in \Real^{l},\, y,\nu \in \Real^{p},
\end{aligned}
\right.
\end{equation}
with $w \sim \mathcal{N}(\mu,I_l )$, and we investigate the construction of a pair of linear systems of smaller, 
possibly minimal, dimension in the state space which have the same stochastic external behavior of, 
and respectively, is equivalent via stochastic bisimulation to $\Sigma$. 
     In the sequel we follow standard practice, see e.g. \cite{ModelChecking} for concurrent processes, \cite{BisimSchaft} for control systems and \cite{PolaIJC06} for switching control systems, and consider relations involved in Definitions \ref{def:ExtEquiv} and \ref{def:stochbisim} that are also equivalence relations on the set of states of $\Sigma$, so that it is
possible to define the quotient of system $\Sigma$ induced by these equivalence relations. 
To this purpose the following result is useful.

\begin{proposition}
\label{propRequiv}
A total relation $\Rcal=\ker(\begin{bmatrix} R_{1} & -R_{2} \end{bmatrix})$ is an equivalence relation on $\mathbb{R}^n$ if and only if $R_1 =R_2$.
\end{proposition}

\begin{proof}
(Sufficiency) If $R_1=R_2$ then the reflexivity, symmetry and transitivity properties are trivially verified.
(Necessity) It is straightforward that if $R_1 \neq R_2$, then the reflexivity property does not hold.
\end{proof}

The following results specialize the geometric conditions derived in the previous section to equivalence relations.

\begin{proposition} \label{prop:bisim}
A total equivalence relation 
\begin{equation}
\label{RbisEquiv}
\Rcal_{\bis}=\ker (\begin{bmatrix} R_{\bis} & -R_{\bis} \end{bmatrix})
\end{equation}
satisfies conditions of Definition \ref{def:stochbisim} 
(i.e., is a stochastic bisimulation with $\Sigma_1=\Sigma_2=\Sigma$ and $R_1=R_2=R_{\bis}$) if and only if
\begin{eqnarray}
&& A\ker(R_{\bis})\subseteq \ker(R_{\bis})\subseteq \ker(C); \label{cond2red}\\
&& \ker(R_{\bis}) \cap \Reach(A,G)=\{0\}. \label{cond3red}
\end{eqnarray}
\end{proposition}

\begin{proof}
Obviously, when $\Sigma_i=\Sigma$ and $R_i=R_\bis$, $i=1,2$, the conditions 
$(h_1)$--$(h_3)$ of Theorem \ref{thm:condRbisimil} are always verified. 
Thus, the sufficiency is proved by showing that \eqref{cond2red} implies conditions $(h_0)$
and $(h_4)$, while \eqref{cond3red} implies $(h_5)$ of Theorem \ref{thm:condRbisimil}.
The implication of property $(h_0)$, i.e.\ $\diag(A,A)$-invariance of $\Rcal$, easily follows 
by considering that
\begin{equation} \label{eq:struttRcal}
\Rcal_\bis = \im\left(\begin{bmatrix} U  &  0  &  \Ubar\\
											           0  &  U  &  \Ubar \end{bmatrix}\right),
\end{equation}
where $U\in\Real^{n\times (n-n_b)}$ and $\Ubar\in\Real^{n\times n_b}$, with $n_b=\rank(R_b)$, 
are matrices such that $\im(U)=\ker(R_\bis)$ and $\rank([U\ \Ubar])=n$.
From \eqref{eq:struttRcal} and the assumption $A\ker(R_{\bis})\subseteq \ker(R_{\bis})$ it 
easily follows that $\diag(A,A)\Rcal_\bis\subseteq\Rcal_\bis$, i.e.\ assumption $(h_0)$.
The implication of $(h_5)$ is trivial.

As far as for the necessity, note that from \eqref{eq:struttRcal} it follows
\begin{equation} 
\im\left(\begin{bmatrix} U \\  0  \end{bmatrix}\right)\subseteq \Rcal_\bis 
\end{equation}
and from assumption $(h_0)$ ($\diag(A,A)$-invariance of $\Rcal_\bis$),
for any $\chi\in\Real^{(n-n_b)}$ we have
\begin{equation} 
\begin{bmatrix} A & 0 \\  0 & A \end{bmatrix} \begin{bmatrix} U \\  0  \end{bmatrix}\chi=
\begin{bmatrix} U  &  0  &  \Ubar\\ 0  &  U  &  \Ubar \end{bmatrix}
\begin{bmatrix} \zeta_1 \\ \zeta_2 \\ \zeta_3\end{bmatrix}\subset\Rcal_\bis
\end{equation} 
for some $\zeta_1,\zeta_2\in\Real^{(n-n_b)}$ and $\zeta_3 \in \Real^{n_b}$.
This implies
\begin{equation} 
\begin{aligned} 
A\, U\, \chi & = U \zeta_1 + \Ubar\zeta_3,\\
 0 & = U \zeta_2 + \Ubar\zeta_3.
\end{aligned} 
\end{equation} 
Since by assumption $[U\ \Ubar]$ is nonsingular, it follows that necessarily 
$\zeta_2=0$ and $\zeta_3=0$, and
therefore $A\, U\, \chi  = U \zeta_1$, which means $A\ker(R_\bis)\subseteq R_\bis$.
Moreover, condition ($h_4$) of Theorem \ref{thm:condRbisimil} 
easily imply $\ker(R_{\bis})\subseteq \ker(C)$. 
Hence, condition (\ref{cond2red}) is true. Condition (\ref{cond3red}) comes from condition 
($h_5$) of Theorem \ref{thm:condRbisimil}.
\end{proof}

\begin{proposition}  \label{prop:Lequiv}
A total equivalence relation 
\begin{equation}  \label{RextEquiv}
\Rcal_{\ext}=\ker \left( 
\begin{bmatrix} R_{\ext} & -R_{\ext} \end{bmatrix}\right)
\end{equation}
satisfies conditions of Definition \ref{def:ExtEquiv}  
with $\Sigma_1=\Sigma_2=\Sigma$ and $R_1=R_2=R_{\ext}$ if and only if
\begin{equation}
A\ker(R_{\ext})\subseteq \ker(R_{\ext})\subseteq \ker(C). \label{cond1red}
\end{equation}
\end{proposition}

The proof follows the same reasoning as used in the proof of Proposition \ref{prop:bisim} and is therefore omitted. \\
Consider any total equivalence relation $\Rcal_{\ext}$ as in (\ref{RextEquiv}) with $\rank(R_{\ext})=n_{\ext}$ and satisfying (\ref{cond1red}) and consider any invertible square matrix 
$\mathbb{T}_{\ext}= [ \mathbb{T}_{\ext}^1 \ \mathbb{T}_{\ext}^2]\in\Real^{n\times n}$
such that $\im(\mathbb{T}_{\ext}^2)=\ker(R_{\ext})$ 
By (\ref{cond1red}), a change of coordinates of $\Sigma$ by means of $\mathbb{T}_{\ext}$ 
gives back system matrices with the following structure
\begin{equation} \label{eq:sysMatExt}
\begin{aligned}
\mathbb{T}^{-1}_{\ext} A \mathbb{T}_{\ext} & =  
\begin{bmatrix} 
A_{\ext}^{11}  &  0 \\
0  &  A_{\ext}^{22} 
\end{bmatrix},&
\mathbb{T}_{\ext} B = 
\begin{bmatrix} B_{\ext}^{1} \\
					      B_{\ext}^{2}  \end{bmatrix}, 	\\[5pt]
C\,\mathbb{T}_{\ext} & = 
\begin{bmatrix} C_{\ext}^{1} & 0 \end{bmatrix}, &
\mathbb{T}_{\ext} G = 
\begin{bmatrix} G_{\ext}^{1} \\
					      G_{\ext}^{2} \end{bmatrix}.
\end{aligned}
\end{equation}
Similarly, consider any total equivalence relation $\Rcal_{\bis}$ as in (\ref{RbisEquiv}) with $\rank(R_{\bis})=n_{\bis}$ 
and satisfying (\ref{cond2red}) and (\ref{cond3red}) and consider any invertible square matrix 
$[\mathbb{T}_{\bis}^1\ \mathbb{T}_{\bis}^2]\in\Real^{n\times n}$
such that $\im(\mathbb{T}_{\bis}^2)=\ker(R_{\bis})$ and 
$\im(\Reach(A,G))\subseteq \im(\mathbb{T}_{\bis}^1)$. 
By (\ref{cond2red}) and (\ref{cond3red}) we get the following structure for the system matrices
\begin{equation} \label{eq:sysMatBis}
\begin{aligned}
\mathbb{T}^{-1}_{\bis} A \mathbb{T}_{\bis} & = 
\begin{bmatrix} 
A_{\bis}^{11}  &  0 \\
0  &  A_{\bis}^{22} 
\end{bmatrix}, &
\mathbb{T}_{\bis} B = 
\begin{bmatrix} B_{\bis}^{1} \\
					      B_{\bis}^{2}  \end{bmatrix},\\[5pt]
C\,\mathbb{T}_{\bis} & = 
\begin{bmatrix} C_{\bis}^{1} & 0 \end{bmatrix}, &				
\mathbb{T}_{\bis} G = 
\begin{bmatrix} G_{\bis}^{1} \\
					      0 \end{bmatrix}.
\end{aligned}
\end{equation}
We now have all the ingredients to define the quotients $\Sigma_{/\Rcal_{\ext}}$ and $\Sigma_{/\Rcal_{\bis}}$ induced by relations $\Rcal_{\ext}$ and $\Rcal_{\bis}$ as follows:
\begin{eqnarray}    \label{SysRedExt}
\Sigma_{/\Rcal_{\ext}}:
\left\{
\begin{aligned}
x_{\ext}(t+1) & =A_{\ext}^{11} x_{\ext}(t) + B_{\ext}^1 u(t) + G_{\ext}^1 w(t) ,\\
y(t) &= C_{\ext}^1 x_{\ext}(t) + \nu(t),\\
x_{\ext} \in & \Real^{n_{\ext}}, \ u \in \Real^{m}, \ w \in \Real^{l}, \ y,\nu \in \Real^{p},
\end{aligned}
\right.\\
													\label{SysRedBis}
\Sigma_{/\Rcal_{\bis}}:
\left\{
\begin{aligned}
x_{\bis}(t+1) & =A_{\bis}^{11} x_{\bis}(t) + B_{\bis}^1 u(t) + G_{\bis}^1 w(t) ,\\
y(t) & = C_{\bis}^1 x_{\bis}(t) + \nu(t), \\
x_{\bis} \in & \Real^{n_{\bis}}, \ u \in \Real^{m},\ w \in \Real^{l}, \ y,\nu \in \Real^{p},
\end{aligned}
\right.
\end{eqnarray}
with $w \sim \mathcal{N}(\mu,I_l )$ and $\nu\sim \mathcal{N}(0,\Psi )$.
The following results then hold.

\begin{theorem}  \label{thm:oderedLequiv}
$\Sigma_{/\Rcal_{\ext}} \cong_{\ext} \Sigma$.
\end{theorem}

\begin{proof}
The proof easily follows by picking a relation $\Rcal$ in the form of (\ref{eq:Rlin}) with $R_1 = I_{n_{\ext}}$ and $R_2 =\begin{bmatrix} I_{n_{\ext}} & 0_{(n-n_{\ext})\times n_{\ext}} \end{bmatrix} \mathbb{T}^{-1}_{\ext}$ 
and verifying by direct computation all conditions of Corollary \ref{cor:Lequiv}.
\end{proof}

\begin{theorem}
\label{thm:oderedbisim}
$\Sigma_{/\Rcal_{\bis}} \cong_{\bis} \Sigma$.
\end{theorem}

\begin{proof}
The proof easily follows by picking a relation $\Rcal$ in the form of (\ref{eq:Rlin}) with $R_1 = I_{n_{\bis}}$ and $R_2 =\begin{bmatrix} I_{n_{\bis}} & 0_{(n-n_{\bis})\times n_{\bis}} \end{bmatrix} \mathbb{T}^{-1}_{\bis}$ 
and verifying by direct computation all conditions $(h_0)$--$(h_5)$ of Theorem  \ref{thm:condRbisimil}.
\end{proof}

We now proceed with a further step by discussing minimal model reduction. Given $\Sigma$ we denote by $\Sigma_{\ext}$ and $\Sigma_{\bis}$ a pair of linear systems of minimal dimension in the state space such that $\Sigma_{\ext} \cong_{\ext} \Sigma$ and $\Sigma_{\bis} \cong_{\bis} \Sigma$. 
It is readily seen that systems $\Sigma_{\ext}$ and $\Sigma_{\bis}$ are unique up to linear transformations. In the sequel we characterize 
$\Sigma_{\ext}$ and $\Sigma_{\bis}$. 
Let 
\begin{equation} \label{Rlinbismax}
\Rcal_{\ext}^{\ast}=\ker(\begin{bmatrix} R_{\ext,1}^{\ast} & -R_{\ext,2}^{\ast} \end{bmatrix}), \quad
\Rcal_{\bis}^{\ast}=\ker(\begin{bmatrix} R_{\bis,1}^{\ast} & -R_{\bis,2}^{\ast} \end{bmatrix})
\end{equation}
be total relations such that $\Rcal_{\ext}^{\ast}$ satisfies conditions $(h_0)$--$(h_4)$ 
(with $\Sigma_i =\Sigma$ and $R_i=R_{\ext,1}^{\ast}$, $i=1,2$) and has maximal dimension,   
while $\Rcal_{\bis}^{\ast}$ satisfies also condition $(h_5)$ of Theorem \ref{thm:condRbisimil}
and is of maximal dimension. 
It is worth mentioning that such relations are not unique in general, as shown in the following simple example.

\begin{example}
Consider system $\Sigma$ as in (\ref{eq:System}) with $A=G=C=I_2$ and any matrix $B$ of compatible dimensions. Then, the two relations in the form of (\ref{eq:Rlin}) 
\[
\ker
\left(
\begin{bmatrix} 
1 & 0 & -1 & 0\\
0 & 1 &  0 & -1
\end{bmatrix}
\right), \quad 
\ker
\left(
\begin{bmatrix} 
1 & 0 & 0 & -1\\
0 & 1 &  -1 & 0
\end{bmatrix}
\right),
\]
are both total and of maximal dimension. 
Moreover, they both satisfy conditions of Theorems \ref{thm:condRequiv} and \ref{thm:condRbisimil} with $\Sigma_1 =\Sigma_2 =\Sigma$.
\end{example}

The above example also shows that in general, $R_{\ext,1}^{\ast}\neq R_{\ext,2}^{\ast}$ and $R_{\bis,1}^{\ast} \neq R_{\bis,1}^{\ast}$. In this case, by Proposition \ref{propRequiv}, $\Rcal_{\ext}^{\ast}$ and $\Rcal_{\bis}^{\ast}$ are not equivalence relations. 
However, it is readily seen that relations 
$\ker ([ R_{\ext,i}^{\ast} \ -R_{\ext,i}^{\ast}])$ and 
$\ker ([ R_{\bis,i}^{\ast} \ -R_{\bis,i}^{\ast}])$ 
with $i=1,2$ still satisfy conditions of Theorems \ref{thm:condRequiv} and  
\ref{thm:condRbisimil}, respectively 
(with $\Sigma_1 =\Sigma_2 =\Sigma$) and by (\ref{kost}) they have the same 
dimensions of $\Rcal_{\ext}^{\ast}$ and $\Rcal_{\bis}^{\ast}$, respectively (and therefore
they are still of maximal dimension). 

Moreover, they are equivalence relations on the set of states of $\Sigma$. 
Hence, without loss of generality we consider in the sequel total relations $\Rcal_{\ext}^{\ast}$ and $\Rcal_{\bis}^{\ast}$ that are also equivalence relations, i.e.\ in the form of 
\begin{equation}
\label{Rredox}
\Rcal_{\ext}^{\ast}=\ker(\begin{bmatrix} R_{\ext}^{\ast} & -R_{\ext}^{\ast} \end{bmatrix}), \quad
\Rcal_{\bis}^{\ast}=\ker(\begin{bmatrix} R_{\bis}^{\ast} & -R_{\bis}^{\ast} \end{bmatrix}).
\end{equation}

We now have all the ingredients to present the main results of this section.
\begin{theorem}
\label{thminredext}
$\Sigma_{\ext} \cong_{l} \Sigma_{/\Rcal_{\ext}^{\ast}}$.
\end{theorem}

\begin{proof}
We first note that $ \Sigma_{/\Rcal_{\ext}^{\ast}} \cong_{\ext} \Sigma$ as an application of Proposition \ref{thm:oderedLequiv} with $\Rcal_{\ext}=\Rcal_{\ext}^{\ast}$. Regarding minimality, suppose by contradiction that $\Sigma_{/\Rcal_{\ext}^{\ast}}$ is not of minimal dimension. Let $n'_{\ext}$ and $n_{\ext}^{\ast}$ be the dimensions of the state spaces of $\Sigma_{\ext}$ and $\Sigma_{/\Rcal_{\ext}^{\ast}}$, respectively. 
Hence, by the contradiction assumption we have $n'_{\ext} < n_{\ext}^{\ast}$. 
By definition of $\Sigma_{/\Rcal_{\ext}^{\ast}}$ we get that $\dim(\Rcal_{\ext}^{\ast})=2n - n_{\ext}^{\ast}$. 
Let $\Rcal'=\ker([R_1' \ -R'_2])$ 
be a total relation satisfying conditions of Definition \ref{def:ExtEquiv} with $\Sigma_1 = \Sigma_{\ext}$ and $\Sigma_2 = \Sigma$ for some matrices $R'_1\in\Real^{r\times n'_{\ext}}$ and $R'_2\in\Real^{r\times n}$. 
  Being $\Rcal'$ total, by \eqref{kost} necessarily $\rank(R'_1)=\rank(R'_2)=n'_{\ext}$, and 
	without loss of generality we can assume $r=n'_{\ext}$, so that
\begin{equation} \label{contra2}
  \rank(R'_2)=r=n'_{\ext}. 
\end{equation}
Consider now the relation $\Rcal''=\ker([R_2'\  -R'_2])$. 
It is easy to see that it satisfies conditions of Definition \ref{def:ExtEquiv} with $\Sigma_1=\Sigma_2=\Sigma$ and from 
\eqref{contra2} and \eqref{kost} we get that $\dim(\Rcal'') = 2n - n'_{\ext}$. 
  Since $n'_{\ext} < n_{\ext}^{\ast}$ then $\dim(\Rcal'')= 2n - n'_{\ext} > 2n - n_{\ext}^{\ast}= \dim(\Rcal_{\ext}^{\ast})$ which contradicts the definition of $\Rcal_{\ext}^{\ast}$ (relation of maximal dimension).
\end{proof}

\begin{theorem} \label{thminredbis}
$\Sigma_{\bis} \cong_{l} \Sigma_{/\Rcal_{\bis}^{\ast}}$.
\end{theorem}

\begin{proof}
We first note that $ \Sigma_{/\Rcal_{\bis}^{\ast}} \cong_{\bis} \Sigma$ as an application of Proposition \ref{thm:oderedbisim} with $\Rcal_{\bis}=\Rcal_{\bis}^{\ast}$. The proof of minimality follows the same steps as in the proof of Theorem \ref{thminredext} and is therefore omitted.
\end{proof}

Provided that one can compute relations $\Rcal_{\ext}^{\ast}$ and $\Rcal_{\bis}^{\ast}$, by using the above results, systems $\Sigma_{\ext}$ and $\Sigma_{\bis}$ are completely specified. The following result fully characterizes $\Rcal_{\ext}^{\ast}$.

\begin{theorem}
\label{thmaxrelext}
A total relation $\Rcal_{\ext}^{\ast}$ as in \eqref{Rredox} is obtained with $R_{\ext}^{\ast}=\Obs(A,C)$.
\end{theorem}

\begin{proof}
The results in Theorem \ref{thm:condRequiv} and \ref{cor:maximalRequiv} state
that any relation that ensures stochastic external equivalence between two system $\Sigma_1$ and $\Sigma_2$ is necessarily
a subspace of $\ker([\Obs(A_1,C_1)\ -\Obs(A_2,C_2)])$, and therefore the maximal of such relations
coincides with $\ker([\Obs(A_1,C_1)\ -\Obs(A_2,C_2)])$ (see also Remark \ref{rem:condstochextequiv}).
These results, particularized for $\Sigma_i=\Sigma$, $i=1,2$, prove that 
the maximal relation that ensures the stochastic external equivalence of a system
with itself is $\ker([\Obs(A,C)\ -\Obs(A,C)])$, which is the thesis.
\end{proof}

By combining Theorems \ref{thminredext} and \ref{thmaxrelext} we get that $\Sigma_{\ext}$ can be chosen as the observable sub--system of $\Sigma$ and it can be easily computed via the Kalman decomposition \cite{Kailath}, by choosing
a nonsingular matrix $\mathbb{T}_{\ext}= [ \mathbb{T}_{\ext}^1 \ \mathbb{T}_{\ext}^2]\in\Real^{n\times n}$
such that $\im(\mathbb{T}_{\ext}^2)=\ker(\Obs(A,C))$.
A system $\Sigma_{\ext}$ of smallest dimension which has equivalent stochastic external behavior of $\Sigma$ is
$\Sigma_{/\Rcal_{\ext}}$ given by \eqref{SysRedExt}, whose dimension is $n-\rank\big(\ker(\Obs(A,C))\big)$.

We now discuss the computation of $\Rcal_{\bis}^{\ast}$.
The computation of the maximal bisimulation relation is generally done through fixed--points operators, see e.g.\ \cite{ModelChecking,BisimSchaft}. In particular, in \cite{BisimSchaft} an algorithm is proposed which converges in finite steps to the desired maximal bisimulation relation. Crucial in this approach is the property of closeness of bisimulation relations with respect to sum of subspaces. 
Unfortunately, this approach cannot be used here because similar closeness properties do not hold. The following example clarifies this issue.
\begin{example}
Consider a linear system $\Sigma$ as in (\ref{eq:System}), where:
\[
A =
\left[
\begin{array}
{ccc}
\alpha & 0 & 0\\ 0 & \alpha  & 0\\ 0 & 0 & \beta 
\end{array}
\right], G =
\left[
\begin{array}
{c}
1 \\ 0 \\ 0 
\end{array}
\right],
C =
\left[
\begin{array}
{ccc}
0 & 0 & 1 
\end{array}
\right],
\]
with $\alpha,\beta\in\Real$ and matrix $B$ of compatible dimensions. By applying Theorem \ref{thm:condRbisimil} with $\Sigma_1=\Sigma_2=\Sigma$ it is possible to show that the equivalence relations
\[
\begin{array}
{l}
\Rcal_{b}=
\ker
\left(
\begin{bmatrix} 
0 &  0 & 1 &  0 &  0 & -1\\
1 & -1 & 0 & -1 &  1 & 0
\end{bmatrix}
\right),\\
\Rcal'_{b}=
\ker
\left(
\begin{bmatrix} 
1 &  0 & 0 & -1 &  0 & 0\\
0 &  0 & 1 & 0  &  0 & -1
\end{bmatrix}
\right)
\end{array}
\]
are total stochastic bisimulation relations between $\Sigma$ and itself. By a straightforward computation we get
\[
\Rcal_{b}+\Rcal'_{b}=\ker\left(
\begin{bmatrix} 
0 &  0 & 1 & 0 & 0 & -1
\end{bmatrix}
\right)
\]
which is not a stochastic bisimulation relation between $\Sigma$ and itself because condition ($h_5$) of Theorem \ref{thm:condRbisimil} (with $\Sigma_1=\Sigma_2 = \Sigma$) is violated.
\end{example}

Non closeness of stochastic bisimulation relations with respect to sum of subspaces poses serious limitations to the use of fixed--point approaches to compute $\Rcal_{\bis}^{\ast}$ in the general case.
However, Proposition \ref{prop:bisim} provides a characterization of $\Rcal_{\bis}^\ast$ that can be helpful for its 
computation and for the consequent computation of $\Sigma_{/ \Rcal_{\bis}^\ast}\cong_{\bis} \Sigma_{\bis}$. 
   Indeed, by Theorem \ref{thm:oderedbisim} a systems $\Sigma_{\bis}$ of minimal dimension that is equivalent via stochastic bisimulation to $\Sigma$
can be computed by finding a total relation $\Rcal_{\bis}^\ast$ of maximal dimension, and this can be computed by 
finding an $A$-invariant subspace $\Ucal\subseteq\Real^n$ of maximal dimension
such that 
\[
\Ucal\subseteq\ker\big(\Obs(A,C)\big),\quad \text{and}\quad 
\Ucal\cap\im\big(\Reach(A,G)\big)=\{0\}.
\]
(see Proposition \ref{prop:bisim}).

Under few additional assumptions we can provide an explicit expression for a relation $\Rcal_{\bis}^\ast$ of maximal dimension, and therefore
of a system $\Sigma_{\bis}$ of minimal dimension that is equivalent via stochastic bisimulation to $\Sigma$.
Let 
$\mathcal{G}=\im(\Reach(A,G))$, $\mathcal{Q}=\ker(\Obs(A,C))$.
Let $\lambda_k\in\Compl$, $k\in\{1,\dots,\delta\}$, $\delta\le n$,  
denote the eigenvalues of $A$, and let $\mathcal{S}_k\subseteq\Real^n$ denote the 
generalized real eigenspaces associated to $\lambda_k$
(subspaces $\mathcal{S}_k$ are $A$-invariant).
Eigenspace $\mathcal{S}_k$ is said to be:
\begin{itemize}
\item totally reachable (from the noise), if $\mathcal{S}_k \subseteq \mathcal{G}$;
\item totally unreachable (from the noise), if $\mathcal{S}_k \cap \mathcal{G} = \{0\}$;
\item totally observable, if $\mathcal{S}_k \cap \mathcal{Q} = \{0\}$;
\item totally unobservable, if $\mathcal{S}_k \subseteq \mathcal{Q}$.
\end{itemize}

\begin{theorem}  \label{thmaxrelbis}
Suppose that any generalized eigenspace $\mathcal{S}_k$ of matrix $A$ 
is either totally reachable or totally unreachable,
and it is either totally unobservable or totally observable, i.e.
\begin{equation}   \label{AssumptH}
\left[[\mathcal{S}_k \subseteq \mathcal{G}] \vee [\mathcal{S}_k \cap \mathcal{G} = \{0\}]\right]
\wedge
\left[[\mathcal{S}_k \subseteq \mathcal{Q}] \vee [\mathcal{S}_k \cap \mathcal{Q} = \{0\}]\right] .
\end{equation}
Let $\Ical=\{k_1,\dots,k_{\bar\delta}\}\subset\{1,\dots,\delta\}$ be a subset of indexes such that
the eigenspaces $\mathcal{S}_{k_i}$, $k_i\in\Ical$, are totally unreachable and totally unobservable
(i.e, $S_{k_i}\subseteq \Qcal$ and $S_{k_i}\cap\Gcal=\{0\}$), and
define $\mathcal{S}_{\Sigma}=\mathcal{S}_{k_1} \oplus \mathcal{S}_{k_2} \oplus ... \oplus \mathcal{S}_{k_{\bar\delta}}\subseteq\Real^n$.
   Then, a total relation $\Rcal_{\bis}^{\ast}$ of maximal dimension takes the form \eqref{Rredox} 
where the (non unique) matrix $R_{\bis}^{\ast}$ is chosen such that $\mathcal{S}_{\Sigma}=\ker(R_{\bis}^{\ast})$.
\end{theorem}

\begin{proof}
Subspace $\mathcal{S}_{\Sigma}$, given as the sum of $A$--invariant subspaces is itself $A$--invariant.
  Moreover, $S_{k_i}\subseteq \Qcal$ and $S_{k_i}\cap\Gcal=\{0\}$ $\forall k_i\in\Ical$
implies that $\mathcal{S}_{\Sigma}\subseteq \mathcal{Q}$ and $\mathcal{S}_{\Sigma}\cap \mathcal{G} = \{0\}$. 
  Hence, $\mathcal{S}_{\Sigma}$ satisfies \eqref{cond2red} and \eqref{cond3red}. 
  Moreover, by its definition, $\mathcal{S}_{\Sigma}$ is of maximal dimension. 
A straightforward consequence of Proposition \ref{prop:bisim} is that a total stochastic bisimulation and equivalence relation $\Rcal_{\bis}$ as in (\ref{RbisEquiv}) between $\Sigma$ and itself is of maximal dimension if and only if $\ker(R_{\bis})$ satisfies (\ref{cond2red}) and (\ref{cond3red}) and is of maximal dimension. Hence, the result follows.
\end{proof}

By combining Theorems \ref{thminredbis} and \ref{thmaxrelbis} we get that if eigenspaces $\mathcal{S}_k$ of matrix $A$ in $\Sigma$ satisfy (\ref{AssumptH}) then $\Sigma_{\bis}$ can be chosen as the sub--system of $\Sigma$ with all and only modes of $\Sigma$ 
that are either reachable from the noise or observable. 
  Such sub--system can be easily computed via the Kalman decomposition \cite{Kailath} with the additional requirement of choosing a complementary space to $\mathcal{G} \cap \mathcal{Q}$ in $\mathcal{Q}$ that is $A$--invariant. 
For the computation of $\Sigma_{\bis}$ it is sufficient to choose an invertible matrix 
$\mathbb{T}_{\bis}=[\mathbb{T}_{\bis}^1\ \mathbb{T}_{\bis}^2]\in\Real^{n\times n}$
such that $\im(\mathbb{T}_{\bis}^2)=\mathcal{S}_{\Sigma}$ and 
$\im(\mathbb{T}_{\bis}^1)=\mathcal{S}_{h_1} \oplus \mathcal{S}_{h_2} \oplus ... \oplus \mathcal{S}_{h_{\delta -\bar\delta}}$,
where $\{h_1,\dots,h_{\delta -\bar\delta}\}=\{1,\dots,\delta\}\setminus \Ical$,
and to use matrix $\mathbb{T}_{\bis}$ to change the coordinates to $\Sigma$ as in \eqref{eq:sysMatBis},
and select the subsystem \eqref{SysRedBis}.

Although the assumption of Theorem \ref{thmaxrelbis} is not demanding
(for instance, it is trivially fulfilled for all systems $\Sigma$ such that 
the matrix $A$ has distinct eigenvalues), it can be 
weakened as follows:

\begin{theorem}  \label{thmaxrelbis_simp}
Suppose that any generalized eigenspace $\mathcal{S}_k$ of matrix $A$ can be decomposed as
$\Scal_k=\Scal_k^1\oplus \Scal_k^2$ with both $S_k^1$ and $S_k^2$ $A$-invariant
and $\Scal_k^1$ such that
\begin{equation} \label{eq:propS1k}
\Scal_k^1\subseteq\Qcal,\quad \Scal_k^1\cap\Gcal=\{0\}
\end{equation}
(note that $S_k^1$ or $S_k^2$ can vanish for some $k$).\\
Define 
\begin{equation} 
\Scal_\Sigma^1=\Scal_1^1\oplus\Scal_2^1\oplus\cdots\oplus\Scal_\delta^1.
\end{equation}
Then, a total relation $\Rcal_{\bis}^{\ast}$ of maximal dimension takes the form \eqref{Rredox} 
where the (non unique) matrix $R_{\bis}^\ast$ is chosen such that $\Scal_\Sigma^1=\ker(R_{\bis}^\ast)$.
\end{theorem}

\begin{proof}
As in the proof of Theorem \ref{thmaxrelbis}, $\Scal_\Sigma^1$, is a sum of $A$-invariant subspaces
satisfying \eqref{eq:propS1k},
and therefore is itself $A$--invariant, and satisfies \eqref{eq:propS1k},
i.e.\ $\Scal_\Sigma\subseteq \Qcal$ and $\Scal_\Sigma^1\cap \Gcal = \{0\}$,
and hence satisfies \eqref{cond2red} and \eqref{cond3red} and is of maximal dimension. 
The same discussion made in the proof of Theorem \ref{thmaxrelbis} leads to the thesis.
\end{proof}

Note that the assumption of Theorem \ref{thmaxrelbis_simp} is weaker than that of Theorem \ref{thmaxrelbis}, 
in that the assumption \eqref{AssumptH} coincides with the assumption made in Theorem \ref{thmaxrelbis_simp}
with the additional condition that for each $k\in\Ical$ either $S_k^1=\{0\}$ or $S_k^2=\{0\}$
(for $k_i\in\Ical$ we have $\Scal_{k_i}^1=\Scal_{k_i}$ and $\Scal_{k_i}^2=\{0\}$).

\section{Connection with Related Literature} \label{sec7}
In this section we establish connections with the notions of bisimulation equivalence given for probabilistic chains and Markov processes 
and with stochastic linear realization theory. 

\medskip

\textit{Bisimulation equivalence for probabilistic chains and Markov processes. } Definition \ref{defstochbisimnondeg} has been inspired by the notion of probabilistic bisimulation given for probabilistic chains in \cite{Larsen:91}. The notion of \cite{Larsen:91} (corresponding to Definition 3.5.3 of \cite{HermannsBook}) has been extended in 
Definition 2.5 of \cite{Desharnais2004} (see also Definition 2.6 of \cite{Desharnais2000}) to labelled Markov processes featuring continuous state space. Definition 2.5 of \cite{Desharnais2004} coincides with Definition 3.5.3 of \cite{HermannsBook} except for the fact that it applies not to equivalence classes but to measurable $\mathcal{R}$--closed sets; we recall that a set $X$ is $\mathcal{R}$--closed if $\mathcal{R}(X)\subseteq X$; if $\mathcal{R}$ is reflexive, this becomes $\mathcal{R}(X) = X$. As also pointed out in \cite{Desharnais2000}, if $\mathcal{R}$ is an equivalence relation, a set is $\mathcal{R}$--closed if only if it is a union of equivalence classes. 
The key difference between Definition 2.5 of \cite{Desharnais2004} and Definition \ref{defstochbisimnondeg} is that while the former considers only $\mathcal{R}$--closed sets measurable sets, the latter considers any measurable set. By Corollary \ref{cor:nondeg}, linear systems with non--degenerate disturbances cannot be reduced no smaller ones while preserving equivalence via stochastic bisimulation. Reduction is possible in the case of linear systems with degenerate disturbances, as shown in Theorem \ref{thminredbis}. 
Definition 2.5 in \cite{Desharnais2004} of bisimulation for labelled Markov processes, instead, allows finding equivalent states even in the case of non--degenerate disturbances. This is a consequence of the fact that conditions in Definition \ref{defstochbisimnondeg} are required to hold for all measurable sets and not only for $\mathcal{R}$--closed sets as in \cite{Desharnais2004}. On the other hand, when Definition 2.5 in \cite{Desharnais2004} is used for reduction purposes, one gets that the original labelled Markov process and the labelled Markov process obtained by aggregating equivalent states are not bisimilar (in the sense of Definition 2.5 in \cite{Desharnais2004}), whereas in our framework linear $\Sigma$ and its reduced one $\Sigma_{/\mathcal{R}}$ are equivalent via stochastic bisimulation, as formally shown in Theorem \ref{thm:oderedbisim}. When relaxing conditions in Definitions \ref{defstochbisimnondeg} and \ref{def:stochbisim} to hold not for all measurable sets but only for $\mathcal{R}$--closed sets (as in Definition 2.5 of \cite{Desharnais2004}), meaning in our framework that conditions in (i) and (ii) are requested to hold only for measurable sets $X_1$ and $X_2$ satisfying
\[
\mathcal{R}^{-1}(\mathcal{R}(X_1))=X_1, \quad \mathcal{R}(\mathcal{R}^{-1}(X_2))=X_2,
\]
Definitions \ref{defstochbisimnondeg} and \ref{def:stochbisim} coincide. 
Consequently, geometric conditions in Theorem \ref{thm:condRbisimil} change. More specifically, conditions ($h_0$)--($h_4$) are still needed while condition $(h_5)$ is not. In fact, if $X_1$ is $\mathcal{R}$--closed,  equality in (\ref{cond511}) would be true independently from condition ($h_5$) (and Lemma \ref{prop:stochbisim}). As a consequence, conditions in Theorems \ref{thm:condRequiv} and \ref{thm:condRbisimil}   would coincide, meaning that Definitions \ref{def:ExtEquiv}, \ref{defstochbisimnondeg} and \ref{def:stochbisim} would coincide, as well. 

\medskip

\textit{Stochastic linear realization theory. } 
	Stochastic linear realization problems deal with realizing a stationary zero mean Gaussian stochastic process through a stochastic linear system $\Sigma$ in the form of \eqref{eq:systemSi} with no control inputs, see e.g. \cite{Kalman1,Faurre} and also \cite{shaftrealization}. 
    For the output process 
$\yb(t,x^0,\ub,\wb,\nub)$ generated by $\Sigma$ with $\ub= \mathbf{0}$ to be stochastically equivalent to the given zero-mean process, the disturbances $\wb$ and $\nub$ must have zero mean.
   Moreover, stationarity of the process implies that it can be realized with an
asymptotically stable system, where the initial condition $x^0$ is a zero mean random vector with covariance $\Psi_x$ satisfying the steady state condition
\begin{equation} \label{eq:statpsix}
\Psi_x=A \Psi_x A^T + GG^T.
\end{equation}
A matrix $\Psi_x$ satisfying \eqref{eq:statpsix} can be computed as
\begin{equation} 
\Psi_x = \sum_{k=0}^\infty A^{k}G G^T (A^T)^{k},
\end{equation}
and is the unique solution if the pair $(A,G)$ is reachable. 
For the sake of generality, we point out that an unstable system can realize 
a stationary stochastic process if the observable subsystem is stable
or if the unstable observable subsystem is not reached by the noise and not excited
by the initial condition. By restricting our attention to stable systems we can state the following:

\begin{proposition}
Two stable linear systems $\Sigma_1$ and $\Sigma_2$ as in \eqref{eq:systemSi} with $\ub_i= \zerob$
and $\mub_i=\zerob$, $i=1,2$, realize the same zero mean Gaussian stochastic process
if and only if they satisfy conditions $(p_0)$ and $(p_3)$ of Proposition \ref{prop:necessequiv1}.
\end{proposition}

As a consequence, if $\Sigma_1$ and $\Sigma_2$ with $\mu_i=\zerob$, $i=1,2$, 
and $\Psi_1=\Psi_2$, are such that $\Sigma_1 \cong_{\ext} \Sigma_2$, 
then they realize the same zero mean Gaussian stochastic process, 
while it is readily seen that the converse implication is not true. 

\section{Conclusions and outlook} \label{sec8}

In this paper we proposed novel definitions of equivalence via stochastic bisimulation and of equivalence of stochastic external behavior for the class of discrete--time stochastic linear control systems with possibly degenerate disturbance distributions.
Necessary and sufficient conditions based on geometric control theory to check these notions were derived and model reduction addressed.  Connections with stochastic reachability and stochastic linear realization theory were also discussed.  \\
In many real world applications, complex systems are given as the composition of several sub--systems. 
In our future work we plant to extend the results presented in this paper to compositional stochastic systems. 
Useful insights in this regard are reported in the last paper by J.C. Willems \cite{Willems13}. 

\bibliographystyle{plain}
\bibliography{biblio1}

\end{document}